\documentclass[a4paper,reqno,10pt]{amsart}
\usepackage[utf8]{inputenc}

\usepackage{enumitem}	
\usepackage{amsmath,amsthm,amssymb,amscd}
\usepackage{amsmath}
\usepackage[all]{xy}
\usepackage[ngerman]{babel}

\theoremstyle{plain}

{

        \newtheorem{thm}{Theorem}[section]

        \newtheorem{prop}[thm]{Proposition}

        \newtheorem{ex}[thm]{Example}
}

{

        \newtheorem{defn}[thm]{Definition}
        
        \newtheorem{rem}[thm]{Remark}

         \newtheorem{conv}[thm]{Convention}
}

\newcommand{\NN}{{\mathbb N}}
\newcommand{\QQ}{{\mathbb Q}}

\newcommand{\ZZ}{{\mathbb Z}}
\newcommand{\RR}{{\mathbb R}}
\newcommand{\CC}{{\mathbb C}}
\newcommand{\PP}{{\mathbb P}}

\newcommand{\T}{{\bf T}}
\renewcommand{\S}{{\bf S}}
\DeclareMathOperator{\SL}{SL}
\DeclareMathOperator{\GL}{GL}

\DeclareMathOperator{\Sp}{Sp}
\DeclareMathOperator{\SO}{SO}
\DeclareMathOperator{\Sym}{Sym}
\DeclareMathOperator{\MC}{MC}

\DeclareMathOperator{\rk}{rk}
\DeclareMathOperator{\id}{id}

\title{On the construction of Calabi-Yau operators}
\author{
Stefan \textsc{Reiter}}
\address{Stefan Reiter, Department of Mathematics,
University of Bayreuth,
95440 Bayreuth,
Germany;
}
\email{stefan.reiter@uni-bayreuth.de}

\begin{document}


\maketitle

\begin{abstract}
 Given a differential operator of geometric origin there exists a list of operations
 that preserve this property, e.g., tensor products, pull-backs, push-forwards and the middle convolution.
 We apply certain sequences of these operations
 to construct  known and new examples of Calabi-Yau operators.
\end{abstract}

\section{Introduction}

Given a one parameter family of Calabi-Yau threefolds the periods of the unique holomorphic diﬀerential 3-form
satisfy a geometric differential equation (Picard-Fuchs equation)  that
fulfills a series of  (arithmetic) properties \cite[Calabi-Yau operators]{vS18}, some of them  encoding invariants of
the underlying family.
The most famous example is the Dwork quintic $\sum_{i=1}^5 x_i^5=x^{-1/5} \prod_{i=1}^5 x_i\;(x \in \CC)$
and the corresponding Picard-Fuchs operator
is the hypergeometric differential operator
$\vartheta^4-5x (\vartheta+1/5)(\vartheta+2/5) (\vartheta+3/5)(\vartheta+4/5), \quad \vartheta=x \frac{d}{dx}$.
From the list of those operators one has extracted a bunch of common features that hopefully describe (comprise)
these.
Thus one calls an irreducible operator $L$ of order $4$ a {\it Calabi-Yau} operator, if, see
\cite[Calabi-Yau operators]{vS18},
\begin{enumerate}
 \item it is irreducible of order $4$
 \item it is of Fuchsian type
 \item it is self-dual
 \item it has $0$ as MUM-point (maximal unipotent monodromy)
 \item it satisfies integrality conditions:
 \begin{enumerate}
  \item the holomorphic solution $y_0 (x) \in  \ZZ[[x]]$
  \item the $q$-coordinate $q(x) \in \ZZ[[x]]$, where $q(x) = \exp(y_1(x)/y_0(x))$ and  $y_1(x) := log(x)y_0(x) + f_1(x), f_1(x) \in x \QQ[[x]]$, is the normalised solution  that contains a single logarithm.
  \item  the instanton numbers $n_d \in \ZZ$ (that can be identified with the genus zero Gopakumar-Vafa invariants) that appear in Lambert-series of
         the Yukawa coupling
         $Y(q)={\rm const} + \sum_{d=1}^\infty n_d d^3 \frac{q^d}{1-q^d}$.
\end{enumerate}
\end{enumerate}

Generally one weakens the integrality conditions to $N$-integrality conditions by allowing mild denominators, see
\cite[Calabi-Yau operators]{vS18}.
For the history and a detailed overview of further amazing aspects of Calabi-Yau operators we refer to the survey article \cite{vS18}.

A classification of Calabi-Yau operators  would give an upper bound for
one parameter families of Calabi-Yau threefolds and would be  a first step their classification.
Meanwhile there are around 550 Calabi-Yau operators collected in  \cite[database]{AESZ}
 {\it https://cydb.mathematik.uni-mainz.de/} which was initiated
 by Almkvist, van Enckevort, van Straten and  Zudilin \cite[preprint]{AESZ}.
 However  some of them were found  in a computer search and  it is still unknown whether they
 are even of  geometric origin, i.e. whether they are factors of Picard-Fuchs equations.

It turns out  that many  examples of Calabi-Yau operators arise from
Hurwitz products (additive convolution) and Hadamard products (multiplicative convolution)
of geometric operators of smaller order (Gauss hypergeometric operators or geometric Heun operators),
s. e.g., Table of Calabi-Yau operators \cite[preprint]{AESZ} and \cite{vS18}.
By the Riemann-Hilbert correspondence a $n$-th order fuchsian differential operator is uniquely determined by its monodromy representation
\[\pi_1(\PP^1(\CC)\setminus \{t_1,\ldots,t_r\},x_0)=\langle  f_1,\ldots,f_{r} \mid f_1\cdots f_r=1\rangle \to \GL_n(\CC),\quad f_i \to T_{t_i}.\]
Hence this enables a  group theoretic approach to construct Calabi-Yau operators, since
the powerful machinery of the additive and multiplicative convolution \cite{Katz96}
can be expressed in terms of the change of the local monodromies.
Moreover, if a monodromy tuple $\T=(T_{t_1},\ldots,T_{t_r}), T_{t_1}\cdots T_{t_r}=\id,$ of a Calabi-Yau operator is symplectically rigid, i.e.
\[ \sum_{i=1}^r C_{\Sp_4} (T_{t_i})=(r-2)\dim \Sp_4=10 (r-2)   \]
then Bogner and the author showed in \cite{BR13}  that this Calabi-Yau operator can be obtained by a sequence of convolutions and tensor products from order one operators.
A similar approach involving also push-forwards in a subsequent paper \cite{BR17} was successful for many cases where
the rigidity condition was slightly weakened
\[ 2+\sum_{i=1}^r C_{\Sp_4}(T_{t_i})=(r-2)\dim \Sp_4=10 (r-2). \]
In most cases one could reduce the question of existence of such a Calabi-Yau operator  to the question of existence of a geometric second order operator.
But even the geometric origin of  second order operators being not hypergeometric, e.g., in the case of  Heun operators,
is in general  an open question.
Nevertheless this approach also produced some new examples for the database.\\

We
impose our candidates for the Calabi-Yau operators only some of the above mentioned properties, i.e., that they are formally self adjoint irreducible fourth order fuchsian differential  operators with
at least a  MUM point (maximal unipotent monodromy).
But we  also require that $L \in \QQ[x][\vartheta], \vartheta:=\frac{dx}{x},$  the operator is of geometric origin
and
that for its (symplectic) monodromy group $G$ holds $(G:(G\cap \Sp_4(\ZZ)) <\infty$.
This seems to be sufficient to obtain after a suitable Moebius transformation also the other $(N-)$integrality conditions.
E.g.  we checked   the first 50 (normalized) instanton numbers for $N$-integrality.

In Section~\ref{MC} the article starts with a review of the main properties of the convolution that we frequently
use in this paper.

Refining the above mentioned  methods in \cite{BR13} and \cite{BR17} we obtain in Section~3 series of old and new examples, some of them in a uniform way,
i.e.  we get families of operators that become by specializing local exponents, the singular points resp.,
Calabi-Yau operators.

In Section~4 we exploit the classification of almost Belyi maps with five exceptional points \cite{vHoeijKu19}.
This enables us to construct in a similar way as in \cite{BR17}  many new examples.

In the next section we study the additive self convolution and obtain different constructions for
a series of known Calabi-Yau operators and additionally produce some new examples.
It is remarkable that starting with certain  hypergeometric operators up to order 6
there appear operators up to order $21$
in the construction process of  Calabi-Yau operators.

Finally we review the some constructions involving the multiplicative convolution.
In the so-called Beauville case we obtain a new example that
was missed in \cite{AZ06}.\\

All in all we contribute $33$ new operators to the database.
We hope these and the new constructions for some known Calabi-Yau operators may enable the experts at least in some cases a construction for a corresponding family
of Calabi-Yau varieties.
A geometric realization of all  Calabi-Yau
operators with symplectically rigid monodromy tuple \cite{BR13} and those constructed in \cite{BR17} having a non symplectically rigid monodromy tuple with four regular singular points
was given in \cite[Summary of results]{DM19}.

\begin{conv}
We will frequently use the following notation
\[ L=[P_0(X),\ldots,P_r(X)] \]
for the operator
\[ L=P_0(\vartheta)+xP_1(\vartheta)+\dots+x^rP_r(\vartheta) \in \CC[\vartheta], \;\vartheta=x\frac{d}{dx}. \]
When scaling an operator
we apply
a multiplication of the solutions of $L$ by algebraic function of the form $(x-t)^k$ that changes the local exponents $a_i$ in the Riemann scheme at $t$
to $a_i+k$ and  the local exponents $b_i$  at $\infty$ to  $b_i-k$.
Thus the local monodromies just changes by scalars.
We also   describe the effect of the Moebius transformations in terms of the Riemann scheme
instead of writing the new operator.

For operators already in the database \cite{AESZ} we mostly construct the Calabi-Yau operators  only up to a Moebius transform and indicate the intermediate steps in the construction by illustrating them via the Riemann scheme.
We use the new numbering in \cite[database]{AESZ}, for some operators we additionally refer to the old AESZ numbering.
If we find new operators we also write down the final Calabi-Yau operator.
Note that so far all known Calabi-Yau operators can be distinguished by the (normalized) instanton numbers $n_1$ and $n_3$ to which we often refer.
To compute the operators in this paper we employed the packages DEtools and Gfun \cite{SZ94}   in Maple.
\end{conv}

\section{Properties of the middle convolution}\label{MC}

We recall some basic properties of the additive and multiplicative middle convolution, see e.g. \cite{DR99}, \cite{DJ17}, \cite{Osh12}, \cite{BR13}, that we frequently employ.\\

Let $f$ and $g$ be  solutions of the irreducible fuchsian differential operators $L_1$ with monodromy tuple $\T$ and $L_2$  with monodromy tuple $\S$, resp..
Then both the additive convolution $\int f(x)g(y-x)dx$ and  the multiplicative convolution (Hadamard product)
$\int f(x)g(y/x)dx$ satisfy again a fuchsian differential equation.
There is a certain subfactor that we call the middle convolution of $L_1$ and $L_2$ that we denote by $L_1 \star L_2$.
This operator corresponds to the middle convolution $\T \star \S$ of the corresponding monodromy tuples (see \cite{DJ17}).
The order of the operator $L_1\star L_2$
is given by the dimension of the corresponding parabolic (group) cohomology
$H_{par}^1(\PP^1(\CC) \setminus \{i, y-j,\infty \mid i \in I\subset \CC, j\in J\subset \CC)\}, {\mathcal{L}})$, where ${\mathcal{L}}={\mathcal{L}}_1 \otimes {\mathcal{L}}_2(y-x)$ in the additive case and
${\mathcal{L}}={\mathcal{L}}_1 \otimes {\mathcal{L}}_2(y/x)$ in the multiplicative case. i.e.
\begin{eqnarray*} \mbox{ order }(L_1\star L_2)&=&\rk(\S) \sum_{\infty\neq i\in I } \rk(T_i-\id)  +\rk(\T) \sum_{\infty \neq j \in J } \rk(S_j-\id)\\
    &&+
\rk(T_{\infty} \otimes S_{\infty}-\id)-2 \rk(\S)\rk(\T)
\end{eqnarray*}
in the additive case and in the multiplicative case (Hadamard product)
\begin{eqnarray*} \mbox{ order }(L_1\star L_2) &=& \rk(\S) \sum_{i\in I \setminus \{ 0, \infty \} } \rk(T_i-\id)  + \rk(\T) \sum_{j\in J\setminus \{ 0, \infty \}} \rk(S_j-\id) \\
&&+
 \rk(T_{\infty} \otimes S_{0}-\id) +\rk(T_{0} \otimes S_{\infty}-\id)   -2\rk(\S)\rk(\T).
 \end{eqnarray*}
The new  singularities are
\[ \{ i+j \mid \infty\neq i\in I, \infty\neq j\in J\} \cup \{ \infty\} \]
in the former case and in the latter case
\[ \{ i j \mid i \in I \setminus \{ 0, \infty \}, j\in J\setminus \{ 0, \infty \}\} \cup \{ 0,\infty\} \]
plus possibly further apparent singularities, i.e. those with trivial monodromy.

Due to Poincar{\'e} duality the middle convolution of two operators having   orthogonal (symplectic) monodromy group
each gives rise to an operator with symplectic monodromy group while the
 middle convolution of an operator with  orthogonal monodromy group and an operator with  symplectic monodromy group
gives rise to an operator with orthogonal monodromy group.

 When working with  monodromy tuples we write $\MC_\alpha(\T)$  for the additive middle convolution
 $\T \star \S$, where $\S=(S_0,S_\infty)=(\alpha,\alpha^{-1})$ is the monodromy tuple of the Kummer sheaf $x^a, \alpha=\exp(2 \pi i a)$.
 We will mainly make use of $\MC_{-1}$, the additive middle convolution with the orthogonal monodromy tuple
 $\S=(S_0,S_\infty)=(-1,-1)$.

\section{Pullback Constructions}

It is well known that a fuchsian hypergeometric differential operator of order $n$
with quasi-unipotent local monodromy is coming from geometry (being  a $n-1$-fold Hadamard product of $n$
geometric order one operators).
Thus we  start our construction with  a fuchsian hypergeometric differential operator of order four
or five having an orthogonal monodromy group.
Applying suitable rational pullbacks followed by  the middle convolution operator $\MC_{-1}$
we obtain  Calabi-Yau operators.

\subsection{Pullbacks of order four hypergeometric operators}

Since all proofs in this section follow the same reasoning and we are ultimately  interested in the
Calabi-Yau operators we sketch just the  proof of Prop.~\ref{p25c} 3) in this section.

\begin{prop}\label{p25c} 
Let $L$ be the  generalized hypergeometric orthogonal differential operator with Riemann scheme
\[  \left\{\begin{array}{c| cccc}
    0 &  0 & 0 & 0 &1/2 \\
    1&  0 & 1/2 &  1 & 2 \\
    \infty &  1/5 & 2/5 &  3/5 & 4/5 \\
 \end{array} \right\}.\]

\begin{enumerate}
 \item  
 Applying  the  pullback
$\frac{(5x-3)^2}{4 x^5}$ with branch pattern $(1^3,2), (3,2), (5)$ (at $0, 1$ and $\infty$)
and a Hadamard product with $(1-x)^{1/2}$
yields after a Moebius transformation the  Calabi-Yau operator with  Riemann scheme
 \[  \left\{\begin{array}{c| cccc}
    0 &  0 & 0 & 0 &0 \\
    1/12 & 0 & 1/2 & 1/2 & 1 \\\
       t_i, i=1,2,3&  0 & 1 &  1 & 2\\
     \infty &  1/2 & 3/2 & 7/2 & 9/2
   \end{array} \right\},\]
  \[3888 \prod_{i=1}^3 (x-t_i)  = 1800 x^3-1200x^2-40x-1, \]
and  instanton numbers $n_1=-4, n_3=3856/9$.
This  is operator \# 5.99 in \cite{AESZ}.

 \item  
Applying  the  pullback
$(5x-4)/x^5$ with branch pattern $(1,4), (1^3,2), (5)$
and a Hadamard product with $(1-x)^{1/2}$
yields after a Moebius transformation the Calabi-Yau operator with  Riemann scheme
\[  \left\{\begin{array}{c| cccc}
    0 &  0 & 0 & 0 &0 \\
    1/256 & 0 & 1/2 & 1/2 & 2 \\\
       t_i, i=1,2,3&  0 & 1 &  1 & 2\\
     \infty &  1/2 & 3/2 & 7/2 & 9/2
   \end{array} \right\}, \mbox{ where }\]
  \[ 8192000 \prod_{i=1}^3 (x-t_i)  =8192000x^3-115200x^2+565x-1, \]
and (new) instanton numbers $n_1=63, n_3=142715$.

Thus we obtain
\[ L=P_0(\vartheta) +\dots + x^5 P_5(\vartheta) \]
where
\[ \begin{array}{ccc}
  P_0(X)&=& - X^4 \\
  P_1(X-1/2)&=& 1077 X^4+399/2 X^2+45/16\\
P_2(X-1)&=& -2^{10} \cdot 3  (3 X-1) (3 X+1) (17 X^2-2)\\
P_3(X-3/2)&=& 2^{15} \cdot 3 (1060 X^4-1305 X^2+270) \\
P_4(X-2)&=&-2^{24} \cdot 5^2 (2 X-3) (2 X+3) (7 X^2-4) \\
P_5(X-5/2)&=& 2^{32} \cdot 5^3 (X-1) (X-2) (X+2) (X+1) \\
   \end{array}
\]

 \item  

Applying  the  pullback
$\frac{1}{4} (x^2-5x+5)^2 x$ with branch pattern $(1,2^2), (1,2^2), (5)$
and a Hadamard product with $(1-x)^{1/2}$
yields after a Moebius transformation the Calabi-Yau     operator with  Riemann scheme
\[  \left\{\begin{array}{c| cccc}
    0 &  0 & 0 & 0 &0 \\
    1/125&  0 & 1 &  1 & 2 \\
    1/200-\sqrt{5}/1000&  0 & 1/2 &  1/2 & 1 \\
    1/200+\sqrt{5}/1000   &  0 & 1/2 & 1/2 & 1 \\
     \infty &  1/2 & 3/2 & 7/2 & 9/2
   \end{array} \right\}\]
and (new) instanton numbers $n_1=75, n_3=116525$.

Thus we obtain
\[ L=P_0(\vartheta) +\dots + x^5 P_5(\vartheta) \]
where
\[ \begin{array}{ccc}
  P_0(X)&=& 4 X^4 \\
  P_1(X-1/2)&=& -15/4 (60X^2+1)(20X^2+3)\\
P_2(X-1)&=& 2^{5} \cdot 5^4  (95X^4-25X^2+2) \\
P_3(X-3/2)&=& -2^{4} \cdot 3 \cdot 5^7 (100X^4-125X^2+27)  \\
P_4(X-2)&=&2^{7} \cdot 5^{10}  (7 X^2-4) (2X-3)(2X+3) \\
P_5(X-5/2)&=& -2^{10} \cdot 5^{13}  (X-1) (X-2) (X+2) (X+1) \\
   \end{array}
\]

\end{enumerate}
\end{prop}

\begin{proof}
 We demonstrate the construction only in the last case.
By \cite{BH}   one concludes that the monodromy group of $L$ is orthogonal.
The  pullback
$\frac{1}{4} (x^2-5x+5)^2 x$ with branch pattern $(1,2^2), (1,2^2), (5)$
changes the Riemann scheme of $L$  to
\[  \left\{\begin{array}{c| cccc}
    0 &  0 & 0 & 0 &1/2 \\
    4&  0 & 1/2 &  1 & 2 \\
    3/2\pm \sqrt{5}/2&  0 & 1 &  2 & 4 \\
     5/2\pm \sqrt{5}/2&  0 & 0 &  0 & 1 \\
    \infty &  1 & 2 &  3 & 4 \\
 \end{array} \right\}.\]
(The branch pattern at $1$ is $(1,2^2)$ corresponding to
$\frac{1}{4} (x^2-5x+5)^2 x-1=  \frac{1}{4} (x-4)(x^2-3x+1)^2$.
 Hence the new  local exponents at $4$,  $3/2\pm \sqrt{5}/2$ resp., are the local exponents of $L$ at $1$, twice the  local exponents of
$L$ at $1$, resp..)
The Hadamard product with $(1-x)^{1/2}$ yields a fourth order symplectic operator
with  Riemann scheme
\[  \left\{\begin{array}{c| cccc}
    0 &  0 & 0 & 0 &0 \\
    4&  0 & 1 &  1 & 2 \\
     5/2\pm \sqrt{5}/2&  0 & 1/2 &  1/2 & 1 \\
    \infty &  1/2 & 3/2 &  7/2 & 9/2 \\
 \end{array} \right\}.\]
using
 \cite[4.2 Convolution and Hadamard product]{BR13} or direct computation via MAPLE.
The order $4$ follows a priori from  the formula   for the Hadamard product in  Section~\ref{MC} since
$ \rk(\T)=4, \rk(\S)=1, $ \[ \sum_{i\in \{4, 3/2\pm \sqrt{5}/2, 3/2\pm \sqrt{5}/2\} }
\rk (T_i-\id)= 5,\; \rk(S_1-\id) =1,\; \rk(T_\infty S_0-\id)=0 \] and   $\rk(T_0 S_\infty-\id)=3$
and the claim on the type of the monodromy group  from the   Poincar{\'e} duality in Section~\ref{MC}
since both operators  are orthogonal.
The change of the local monodromy is given in \cite[Prop. 2.13]{BR13}.
The appearance of an apparent singularity could be gathered from the Fuchs relation for the local exponents.
Note that from the middle convolution applied to monodromy tuples one can't derive the existence of an apparent
singularity.
Finally the Moebius transformation $x \mapsto 2^2 5^3 x$ yields the Calabi-Yau operator.
\end{proof}

\begin{prop}\label{pbS34}

Starting with
a hypergeometric operator $L_0=[X^3(X-1/2),-(X+a)(X+1-a)(X+1/3)(X+2/3)]$
with Riemann scheme
\[  \left\{\begin{array}{c| cccc}
    0 &  0 & 0 & 0 &1/2 \\
    1&  0 & 1/2 &  1 & 2 \\
    \infty &  a & 1/3 &  2/3 & 1-a \\
 \end{array} \right\},\; a \in \{ 1/2,1/3,1/4,1/6\}, \]
and the pullback  $ 2x(x-3/2)^2$ followed by
 a Hadamard product with $(1-x)^{1/2}$
yields after a Moebius transformation the symplectic operator $L$ with  Riemann scheme
\[  \left\{\begin{array}{c| cccc}
    0 &  0 & 0 & 0 &0 \\
    2&  0 & 1 &  1 & 2 \\
    3/2&  0 & 1/2 &  1/2 & 1 \\
     \infty &  1/2 & 3a & 3-3a & 5/2
   \end{array} \right\}.\]

Thus we obtain
\[ L=P_0(\vartheta) +\dots + x^3 P_3(\vartheta) \]
where
\[ \begin{array}{ccc}
  P_0(X)&=& -18 X^4 \\
  P_1(X-1/2)&=&33 X^4-(81a^2-81a+21/2)X^2+9/16(2a-1)^2\\
P_2(X-1)&=&-20 X^4+(108a^2-108a+28)X^2 -4(3a-1)(3a-2)\\
P_3(X-3/2)&=& (X+1)(X-1)(2X-3+6a)(2X+3-6a)

   \end{array}
\]

For the special values  $a\in \{ 1/2,1/3,1/4,1/6\}$ and a suitable Moebius transformation $x \mapsto v x$ we get Calabi-Yau-operators with (normalized) instanton numbers $n_1$ and $n_3$

\[ \begin{array}{c cccc cc}
      a &   v& n_1 & n_3 &\# \mbox{ in \cite[database]{AESZ}} & \mbox{ remarks }\\
         0   &  2 \cdot 3 & 5 & 1/3 & & \mbox{ reducible} \\
    1/2 & -2^5 \cdot 3  & 4 & 644/3 & 2.52  \\
    1/3 & 2^1 \cdot 3^4 & 9 & 748 & 3.4  \\
    1/4 & 2^7 \cdot 3 & 68 & 18628/3 & 3.3 \\
    1/6 & 2^5 \cdot 3^4 & 900 & -8364884 & 3.2  \\
   \end{array}
 \]
\end{prop}

\begin{prop}
 Let $L$ be the  generalized hypergeometric orthogonal differential operator with Riemann scheme
\[  \left\{\begin{array}{c| cccc}
    0 &  0 & 0 & 0 &1/2 \\
    1&  0 & 1/2 &  1 & 2 \\
    \infty &  1/8 & 3/8 &  5/8 & 7/8 \\
 \end{array} \right\}\]

Applying  the  pullback $(4x-3)/x^4$ with branch pattern $(1,3), (1^2,2), (4)$
gives the Riemann scheme
\[  \left\{\begin{array}{c| cccc}
    0 & 1/2 & 3/2 & 5/2 &7/2 \\
    1 & 0 & 1 & 3 & 4 \\
    3/4&  0 & 0 &  0 & 1/2 \\
    -1\pm \sqrt{2}& 0 &1/2& 1 & 2 \\
    \infty &  0 & 0 &  0 & 3/2 \\
 \end{array} \right\}\]
Multiplication with $x^{-1/2}(x-3/4)^{1/2}$ and a Hadamard product with $(1-x)^{1/2}$ yields
after a Moebius transformation the Riemann scheme
\[  \left\{\begin{array}{c| cccc}
    0 & 0 & 0 & 0 &0 \\
    1/48 & 0 & 1 & 3 & 4 \\
    7/1296\pm 1/324 \sqrt{-2}& 0 &1& 1 & 2 \\
    \infty &  1/2 & 1/2 &  1/2 & 1/2 \\
 \end{array} \right\}\]
with instanton numbers $n_1=20 $ and $ n_3=-119332/9 $.
This is  operator \# 4.72.
Having  a second maximal unipotent monodromy point (MUM-point) at infinity we also get the operator
\# 4.73  with instanton numbers $n_1=2656$ and $n_3=2493879008$.
\end{prop}

\subsection{Pullbacks of order five operators} 

In this section we use the following construction to obtain Calabi-Yau operators.

Let $L_1$ be a second order  hypergeometric operator, Heun operator, resp., with Riemann scheme
\[  \left\{\begin{array}{c| ccc}
    0 & 0 & 0 \\
    1& a & -a \\
    \infty & 1 & 1
   \end{array} \right\},
    \left\{\begin{array}{c| ccc}
    0 & 0 & 0 \\
    1& 0 & 0 \\
    t & 0 & 0 \\
    \infty & 1 & 1
   \end{array} \right\},  \mbox{ resp. }. \]
For     $a\in \{ 1/2,1/3,1/4,1/6\}$  the hypergeometric operator $L_1$ is related to  families of elliptic curves.
If $(t,s)\in \{ (-1,0),(1/2,-1),(2,-1/2)\}$ or
\[(t,s) \in \{ (-8,-1/4),(-1/8,2),(9,-1/3),(1/9,-3),(8/9,-3/4),(9/8,-2/3)\}\]
then the  Heun operator $L_1=[tX^2, -(X^2 (t+1)+X(1+t)-s t), (X+1)^2]$
is a so-called
Beauville operator
also arising from a family of elliptic curves, see \cite{Herf91}, \cite[Elliptic fibre products]{vS18}.
Let further $L_2$ be the hypergeometric operator with Riemann scheme
   \[ \left\{  \begin{array}{c| ccc}
    0 & 0 & 0 \\
    1& 0 & 0 \\
    \infty & 1/3 & 1/3
   \end{array} \right\}.
    \]
The Hadamard product of $L_1$ with $L_2$ gives a fourth order operator with symplectic monodromy group and Riemann scheme
\[ \left\{\begin{array}{c| cccc}
    0 & 0 & 0 & 0 & 0 \\
    1& 0 & a & -a & 1 \\
    \infty & 1/3 & 1/3 & 2/3 &2/3
   \end{array} \right\},
    \left\{\begin{array}{c| cccc}
    0 & 0 & 0 & 0 & 0 \\
    1& 0 &
    1 & 1 & 2 \\
     t& 0 & 1 & 1 & 2 \\
    \infty & 1/3 & 1/3 & 2/3 &2/3
   \end{array} \right\}, \mbox{ resp.}.
   \]

This construction is already used in
\cite[3.1]{Alm06}.

Applying the middle convolution $\MC_{-1}$, e.g., \cite{DR99}, \cite[Section 2]{BR13}, yields
a fifths order operator with orthogonal monodromy group and Riemann scheme
\[  \left\{\begin{array}{c| ccccc}
     0 & 0 & 1/2 & 1/2 & 1/2 & 0 \\
    1& 0 & a+1/2 & 1 & 3/2-a & 2 \\
    \infty & -1/6 & 1/6 & 1/2& 5/6 &7/6
   \end{array} \right\},  \]
\[  \left\{\begin{array}{c| ccccc}
     0 & 0 & 1/2 & 1/2 & 1/2 & 0 \\
    1& 0 & 1 & 3/2 & 2 & 3 \\
     t& 0 & 1 & 3/2 & 2 & 3 \\
    \infty & -1/6 & 1/6 & 1/2& 5/6 &7/6
   \end{array} \right\}, \mbox{ resp.}.\]

A pullback with $S_{34} := (2x+3) x^2$
with branch pattern $(2,1), (2,1), (3)$ at $0, 1$ and $\infty$
changes the Riemann scheme to
\[ \left\{ \begin{array}{c| ccccc}
    -3/2 & 0 & 1/2 & 1/2 & 1/2 & 1 \\
      0& 0 & 1 & 1 & 1 & 2 \\
     1/2&  0 & a+1/2 & 1 & 3/2-a & 2 \\
     -1 & 0 & 2a+1 & 2 & 3-2a & 4 \\
    \infty & -1/2 & 1/2 & 3/2& 5/2 &7/2
   \end{array} \right\},\]
\[ \left\{ \begin{array}{c| ccccc}
    -3/2 & 0 & 1/2 & 1/2 & 1/2 & 1 \\
      0& 0 & 1 & 1 & 1 & 2 \\
     1/2&  0 & 1 & 3/2 & 2 & 3 \\
     -1 & 0 & 2 & 3 & 4 & 6 \\
     t_1 & 0 & 1 & 3/2 & 2 & 3 \\
     t_2 & 0 & 1 & 3/2 & 2 & 3 \\
     t_3 & 0 & 1 & 3/2 & 2 & 3 \\
    \infty & -1/2 & 1/2 & 3/2& 5/2 &7/2
   \end{array} \right\}, \mbox{ resp.},\]
 where we have an apparent singularity at $\infty$ and $  S_{34}-t=2(x-t_1)(x-t_2)(x-t_3)$.

 Hence $MC_{-1}$ and a following Moebius transformation result in
 a fourth order operator $L$ with symplectic monodromy group, e.g., \cite[Thm.~5.14]{DR99}, \cite[Section 2]{BR13}, and Riemann scheme
\[  \left\{\begin{array}{c| cccc}
    0 & 0 & 0 & 0 & 0 \\
    1& 0 & 2a-1/2 &1 &3/2-2a \\
     3& 0& 1/2& 1/2& 1 \\
     4& 0& 1-a& 1 & 1 + a \\
   \infty &1/2& 5/2& 7/2& 11/2
   \end{array} \right\}, \]

\[  \left\{\begin{array}{c| cccc}
    0 & 0 & 0 & 0 & 0 \\
     3& 0& 1/2& 1/2& 1 \\
     4& 0& 1& 1 & 2 \\
     s_1& 0& 1& 1 & 2 \\
     s_2& 0& 1& 1 & 2 \\
      s_3&0& 1& 1 & 2 \\
     \infty &1/2& 5/2& 7/2& 11/2
   \end{array} \right\}, \mbox{ resp.}, \]
where $(x-s_1)(x-s_2)(x-s_3)= 4(x+3/2) (x/2-3/2)^2 -4t.$

In the former case we obtain
\[ L=P_0(\vartheta) +\dots + x^6 P_6(\vartheta) \]
where
\[ \begin{array}{ccc}
  P_0(X)&=&2304X^4 \\
  P_1(X-1/2)&=&-7296X^4+(-1776-9b)X^2+36+1/4b\\
P_2(X-1)&=&8848X^4+(6b-560)X^2-2b-800 \\
P_3(X-3/2)&=&
-5248X^4+(-b+4544)X^2+b+1404 \\
P_4(X-2)&=&
1632X^4-3576X^2-608 \\
P_5(X-5/2)&=&
-256X^4+1040X^2 \\
P_6(X-3) &=&(4X^2-25)(4X^2-1) \\
b&=&576(a-1/2)^2
   \end{array}
\]

and in the latter case

\[ L=P_0(\vartheta) +\dots + x^6 P_6(\vartheta) \]
where
\[ \begin{array}{ccc}
 P_0(X)&=&-2304 t X^4 \\
  P_1(X-1/2)&=&
192 (11 t+27) X^4
-96 (54 s t+7 t+15) X^2
+144 s t+36 t+36
\\
P_2(X-1)&=&
-16 (40 t+513) X^4
+16 (216 s t+56 t+141) X^2
-1152 s t-256 t+192 \\
P_3(X-3/2)&=&
64 (t+81) X^4
-16 (36 s t+13 t+298) X^2+
576 s t+144 t-1116
  \\
P_4(X-2)&=&
-1632 X^4
3576 X^2
608
  \\
P_5(X-5/2)&=&6(16X^2-65)X^2
  \\
P_6(X-3) &=&-(4X^2-25)(4X^2-1) \\
\end{array}\]

For the special values  $a\in \{ 1/2,1/3,1/4,1/6\}$ and a Moebius transformation $x \mapsto v x$ we get  Calabi-Yau-operators
with instanton numbers $n_1$ and $n_3$
in the former case
\[ \begin{array}{c ccccc}
      a &   v& n_1 & n_3 & \#  \\
    1/2 & 2^6 \cdot 3 &-28&-108956/3& \mbox{ new }\\
    1/3 & 2^2 \cdot 3^4& -63 & -195443 &\mbox{ new }\\
    1/4 & 2^8 \cdot 3 & -196 & -9021764/3 &\mbox{ new } \\
    1/6 & 2^6 \cdot 3^4 & -1764 & -1141651916& \mbox{ new } \\
    1   & 2^2 \cdot 3 & -7 & -80/3 & \mbox{ new }
   \end{array}
 \]
 (if $a=1$ then the operator is reducible but produces also integer instanton numbers)
 and
 in the latter case
\[ \begin{array}{c ccccc}
      t & s&   v& n_1 & n_3 & \# \\
     -1 &0 & 2^4 \cdot 3 &-16& -15248/3&\mbox{ new } \\
    1/2&-1 & 2^4 \cdot 3& 8  & 9928/3 & 6.14  \\
    2 &-1/2& 2^5 \cdot 3 & -8 & -3784/3 &6.39 \\
\\

    -8 & -1/4& 2^5 \cdot 3 &  -11 & -3422/3 &6.6 \\
    -1/8&2   & 2^2 \cdot 3 &  -25 & -17452 &6.5\\
      9 & -1/3 &2^2 \cdot 3^3 &  -9 & 217/3 &6.7\\
       1/9 & -3  & 2^2 \cdot 3 &  23 & 16723 & 6.10\\
      8/9& -3/4 & 2^5 \cdot 3 &  -5 & -10994/3& \mbox{ new }\\
    9/8 &-2/3  & 2^2 \cdot 3^3 & -9 & -15092/3 & \mbox{ new }
   \end{array}
 \]

 \section{(almost)-Belyi maps with five exceptional points}

Here we make use of the Table of (almost)-Belyi maps with five exceptional points
by
Mark van Hoeij and Vijay Kunwar \cite{vHoeijKu19}.
We apply an  almost-Belyi map with five exceptional points to one of the following hypergeometric operators with
Riemann scheme
\[  \left\{\begin{array}{c| ccc}
    0 & 0 & 0 \\
    1& -1/4 & 1/4 \\
    \infty & a & 1-a
   \end{array} \right\}\]
where $a \in \{ 1/3,3/8,5/12\}$.
Thus the monodromy group is commensurable to $\SL_2(\ZZ)$.
We choose  those  almost-Belyi maps such that we get after a Moebius transformation an operator $L$ with Riemann scheme of the following type
\[  \left\{\begin{array}{c| ccc}
    0 & 0 & 0 \\
   \pm  1&  -c_1 & c_1 \\
    \pm t  & -c_2 & c_2 \\
     t_1 & 0 & 2
   \end{array} \right\},\]
where $t_1\neq \infty$.
Since for a solution $f$ of $L$ the product $f(x)f(-x)$ is invariant under $x \mapsto -x$,
which corresponds to the tensor product  of $L(x)$ with $L(-x)$, see \cite[Def. 4.3]{BR13},
we can apply the map (push forward) $(x^{1/2})^*$  as in [BR17, Sec. 3] and
receive a fourth operator with orthogonal monodromy group and
Riemann scheme
\[  \left\{\begin{array}{c| cccc}
    0 &  0 & 0 & 0 &1/2 \\
    1&  -2c_1 & 0 &  2c_1 & 1 \\
    t^2&  -2c_2 & 0 &  2c_2 & 1 \\
     t_3 &  0 & 1 & 2 & 4 \\
     t_4 &  0 & 1 & 2 & 4 \\
     \infty &  1 & 3/2 & 3 & 4
   \end{array} \right\}.\]
A Hadamard product with $(1-x)^{1/2}$ yields a fourth operator with symplectic monodromy group and
Riemann scheme
\[  \left\{\begin{array}{c| cccc}
    0 &  0 & 0 & 0 & 0 \\
    1&  -2c_1 +1/2 & 0 &  2c_1+1/2 & 1 \\
    t^2&  -2c_2+1/2 & 0 &  2c_2+1/2 & 1 \\
     t_5 &  0 & 1 & 3 & 4 \\
     \infty &  1/2 & 3/2 & 3/2 & 5/2
   \end{array} \right\}.\]

In the following cases we get Calabi-Yau operators after a Moebius transformation changing the Riemann scheme to

\[  \left\{\begin{array}{c| cccc}
    0 &  0 & 0 & 0 & 0 \\
    \alpha&  -2c_1 +1/2 & 0 &  2c_1+1/2 & 1 \\
    \beta &  -2c_2+1/2 & 0 &  2c_2+1/2 & 1 \\
     \gamma &  0 & 1 & 1 & 2 \\
     \infty &  1/2 & 3/2 & 7/2 & 9/2
   \end{array} \right\},\]
where $\alpha, \beta, \gamma \in \QQ$ or $\gamma \in \QQ, (x-\alpha)(x-\beta) \in \QQ[x]$ and $ c_1=c_2$.
Note that at $\infty$ we have after scaling an apparent singularity.
(We only list operators that do not appear in \cite{BR13}, \cite{BR17}. There are sometimes several almost-Belyi maps
 that give rise
to the same Calabi-Yau operator (we only list one of those)).
In addition starting with the same almost-Belyi  map in  \cite{vHoeijKu19} can result in different Calabi-Yau operators
(depending on the Moebius transformation following it to obtain an operator $L$).
\begin{enumerate}
 \item known cases
 \[ \begin{array}{c cccccccccc}
     \mbox{ \cite{vHoeijKu19} }  & a&c_1 &c_2& n_1 & n_3 & \# & AESZ\\
     N_{55} &1/3 & 0 & 0&1312& 127846048 &5.114=5.118 & 412=416 \\
      N_{60} &1/3& 0 & 0&352& 3284448&5.90& 330\\
    \\
      N_{20}&1/3&  1/4 & 0& 22432 & 425234532128 &5.116 & 414 \\
    \\
      N_{11} & 1/3& 1/3 & 0 &  1083168 &32204207145918624&5.117 & 415 \\

          N_{31}& 1/3 & 1/3 & 0 & 3843&2715123387 &5.115 & 413 \\
          \\
      O_{14} &1/8 & 1/8 & 0 &  80 & 174096& 5.13 & 83
    \end{array}
\]
 \item new cases
 \[ \begin{array}{c cccccccccc}
  \mbox{\cite{vHoeijKu19}} &a& c_1 &c_2& n_1 & n_3 & \\
     N_{56}  & 1/3&0 & 0& 44 &33404   \\
      N_{57} &1/3& 0 & 0& 128736 & 26197767783776\\
       N_{58} &1/3& 0 & 0&885& 47779414 \\
       N_{60a} &1/3& 0 & 0&633760 &  2121862738725664\\
        N_{60b} &1/3& 0 & 0 &433&23061052/3\\
         N_{60c} &1/3& 0 & 0& 28 & 38332/3\\
    \\
      N_{18a} &1/4& 1/4 & 0&103554 & -2002161892680 \\
       N_{18b} &1/4 & 1/4 & 0&1328 & 170147184\\
        P_{5} & 1/4& 1/4 & 0&  184068 & 71387757415212\\
    \\
      N_{32} &  1/3& 1/3 & 0 & 156 &  765460  \\
       N_{33} & 1/3 & 1/3 & 0 & 402564000 &  199695555707527015523104  \\
         N_{34}& 1/3 & 1/3 & 0 &  14176 &  72835177440 \\
       N_{35} & 1/3 & 1/3 & 0 &  176 &  1651440 \\

          \\
      O_{15a} &1/8& 1/8 &0  & 87968 & 6050453456672\\
       O_{15b} &1/8 & 1/8 & 0 & 1564 & 1995940660/9\\
       O_{15c} &1/8  & 1/8 & 0& 58 &930656/9 \\
    \end{array}
\]
\end{enumerate}

The operators can be written as
\[ N:=4 X^4 + x N_{1}(X+1/2)+x^2 N_{2}(X+1) +x^3 N_{3}(X+3/2)+ x^4  N_{4}(X+4/2)+ x^5 N_{5}(X+5/2),\]

where
\begin{eqnarray*}
 N_{1}(X)&=&a_1 (b_1 X^4+c_1 X^2+d_1) \\
 N_{2}(X)&=&a_2 (b_2 X^4+c_2 X^2+d_2) \\
 N_{3}(X)&=&a_3 (b_3 X^4+c_3 X^2+d_3)\\
 N_{4}(X)&=&a_4 (2X-3)(2X+3)(b_4 X^2+c_4) \\
 N_{5}(X)&=&a_5 (X-1)(X+1)(X+2)(X-2)
\end{eqnarray*}

\[\begin{array}{cccc}
N & [a_1,b_1,c_1,d_1] &[a_2,b_2,c_2,d_2]  & \\
 N_{56} & [-1,3088, 600, 9] &  [2^8,3452, -787, 44]&\\
 N_{57} & [-2^{10}, 3619, 439, 27]&[2^{21}, 653648, -207560, 34797] & \\
  N_{58} & [-2^{-2} 3^{11},47664, 7336, 243] &[2^8 3^4, 6084, -1666, 205] &
\\
 N_{60a} & [2^6,270592, 32608, 2209]  &  [2^{ 24},1786624, -570016, 99225] &   \\
  N_{60b} &[-2^{-2},76432, 12472, 361]  & [2^5 3^1, 371524, -97235, 10731] &\\
 N_{60c} &[-2^2, 544, 100, 1]& [2^6  3^1, 2224, -560, 33]&  \\
  N_{18a} &[-3^2, 242352, 16560, 2401]&[2^4 3^{10} ,487404, -202773, 43040]& \\
 N_{18b} &[-2^2, 14096, 2488, 81]&[2^{12}, 75196, -18770, 2119]&   \\
  N_{P5} & [-2^4 3^2, 36612, 4707, 289],&  [2^{11} 3^{10}, 22932, -7073, 1185]& \\
 N_{32} &[-2^2 3^1, 736, 140, 3]& [2^6 3^2, 12816, -2992, 235] &\\
  N_{33} &[-2^6 3^1 5^1,10301184, 1188896, 98415]& [2^{26} 3^6, 199860352, -65431288, 12285285]&\\
 N_{34} &[-2^6, 7264, 1024, 49]& [2^{15} 3^1, 218744, -63866, 9575] & \\
  N_{35} &[-2^2, 2704, 504, 9] & [2^{12},2716, -674, 55] & \\
 O_{15a} &[-2^6, 37376, 4128, 289]& [2^{21}, 271904, -89708, 15633]&  \\
 O_{15b} & [-2^2,15136, 2276, 81]]&[2^6 3^1, 1891504, -528176, 68257]& \\
 O_{15c} &[-1, 4528, 848, 9]&[2^4 3^1, 37508, -10135, 752] &   \\
\end{array}\]
\[\begin{array}{ccccc}
N_x &  [a_3,b_3,c_3,d_3] &[a_4,b_4,c_4] & a_5 \\
 N_{56} &[-2^{11} 7^2, 1132, -1399, 288]& [2^{15} 7^4, 20, -11]& -2^{20} 7^6 \\
 N_{57} &[-2^{37} 3 5^2, 24551, -32318, 8712]&[2^{47} 3^2 5^4 7, 1052, -659] & -2^{62} 3^5 5^6 7^2& \\
  N_{58} &[-2^9 3^8 5, 13068, -16663, 4000] &  [2^{14} 3^{ 13} 5^2, 72, -43] & -2^{18} 3^{18} 5^4& \\
 N_{60a} & [ 2^{ 41} 3^2 5^2, 52400, -69103, 18818]&[2^{56} 3^5 5^4, 256,-11]  &2^{74} 3^8 5^6 \\
  N_{60b} & [-2^7 3^2 5^2 7^2, 22956, -29019, 6728]&[2^9 3^3 5^4 7^4,172,-101]  &- 2^{14} 3^4 5^6 7^6 \\
 N_{60c} &  [-2^{ 16} 3^2, 63, -78, 16]& [2^{18} 3^3, 52,-29 ]  & -2^{28} 3^4 \\
  N_{18a} &  [-2^6 3^{ 18} 7^1, 269244, -382091, 123823]&[2^{12} 3^{26}  7^2, 1110,- 743] &  -2^{20}  3^{34} 7^4\\
 N_{18b} &[-2^{17} 3^1 5^2, 82372, -103621, 23814]& [2^{22} 3^2 5^4 19^1,568, -331 ]& -2^{30} 3^4 5^6 19^2 \\
  N_{P5} &  [-2^{18} 3^{17} 5^1 ,4299, -5629, 1500]&[2^{22} 3^{25} 5^2, 268, -167&   -2^{32}  3^{32} 5^4\\
 N_{32} & [-2^{ 16} 3^3, 1621, -2014, 432]& [2^{18} 3^4 13^1, 460,-259], &-2^{28} 3^6 13^2 \\
  N_{33} &[-2^{45} 3^{12} 13^1, 19881456, -26399045, 7386314]&[2^{62} 3^{18} 5^1 11^1 13^2,17976,-11423]& -2^{82} 3^{24} 5^2 11^2 13^4 \\
 N_{34} & [-2^{27} 3^2 5^2, 16380, -21193, 5408]&[2^{32} 3^3 5^4 17^1,1148,-703],&-2^{46} 3^4 5^6 17^2 \\
  N_{35} &[-2^{17} 5^1, 8284, -10339, 2250] & [2^{22} 5^2 13^1, 232,-133]& -2^{30} 5^4 13^2 \\
 O_{15a} & [-2^{38} 3^2, 27388, -36317, 10000]& [2^{50} 3^5 7^1, 524,-331] & -2^{68} 3^8 7^2 \\
 O_{15b} &  [-2^{16} 3^4 7^2, 4147, -5310, 1296]&[2^{18} 3^3 7^4 19^1,1228,-739]&   -2^{28} 3^4 7^6 19^2\\
 O_{15c} &  [-2^6 3^3 5^1, 35324, -44139, 9375]&[2^{11} 3^3 5^2 7^1,598,-337]&  -2^{18} 3^4 5^4 7^2 \\
\end{array}\]

However not all Calabi-Yau operators with a Riemann scheme as above can be constructed this way,
e.g., see Prop~\ref{p25c} or the cases  \# 5.109 and \# 5.37 in \cite{AESZ}.

\begin{rem}
The Calabi-Yau operator \# 3.2 has the  Riemann scheme
\[  \left\{\begin{array}{c| cccc}
    0 &  0 & 0 & 0 &0 \\
    2&  0 & 1 &  1 & 2 \\
    3/2&  0 & 1/2 &  1/2 & 1 \\
     \infty &  1/2 & 1/2 & 5/2 & 5/2
   \end{array} \right\}\]
and  instanton numbers $n_1=900, n_3=-8364884$, see Prop.~\ref{pbS34}.
This operator can also be obtained by the above construction
from the hypergeometric operator with Riemann scheme
\[  \left\{\begin{array}{c| ccc}
    0 & 0 & 0 \\
    1& -1/4 & 1/4 \\
    \infty & 5/12 & 7/12
   \end{array} \right\}\]
with the pullback $(x-1)^3$ which gives after a Moebius-transformation (and scaling)
the Riemann scheme
\[  \left\{\begin{array}{c| ccc}
    0 & 0 & 0 \\
    I& 1/4 & 3/4 \\
     -I& 1/4 & 7/4 \\
      \pm \sqrt{3}& 0 & 0 \\
   \end{array} \right\}.\]
Hence this operator is not invariant under $x \mapsto -x$ and we can proceed as above.
\end{rem}

\section{Additive Self Convolution}

The differential Galois group of the Laplace-Fourier transform $FT$ of the additive middle convolution $L_1 \star \dots \star L_r$
of the operators
$L_1, \dots, L_r$ corresponds to the tensor product of the differential Galois groups of the  Fourier-Laplace transforms $FT(L_i), i=1,\ldots,r$.
Thus one obtains information about the decomposition of the convolution product $L_1 \star \dots \star L_r$,
especially if $L_1=\ldots=L_r$.

\begin{defn}
 By the $n$-th symmetric,  n-th exterior, resp., power of the self convolution of $L$ we denote  the operator
 $P$ with $FT(P)=\Sym^n(FT(L))$,  $FT(P)=\Lambda^n(FT(L))$, resp..
\end{defn}

\begin{rem}\label{FT}
 If $L$ is an irreducible Fuchsian operator with monodromy tuple $\T$ then
 the order of $FT(L)$ is $m:=\sum_{i \in \CC} \rk(T_i-\id)$.
 It has only two non apparent singularities, a regular at $0$ and an irregular at $\infty$.
 The Jordan form $J(S_0)$ of the monodromy of $FT(L)$ at $0$ can be obtained from the local monodromy $T_\infty$   as follows.
 If $J(T_\infty)=\sum_{\alpha\in \CC}  \sum_{k\in \NN} \alpha J_k^{n_k}$ then
 \[ J(S_0)= \sum_{1\neq \alpha\in \CC}  \sum_{k\in \NN} \alpha J_k^{n_k} \oplus   \sum_{k\in \NN}  J_k^{n_k+1} \oplus J_1^v,\]
 where $J_k^{n_k}$ denotes an unipotent Jordan block of seize $k$ with multiplicity $n_k$ and $v$ is such that $\rk(J(S_0))=m$.
 Hence one can easily determine the local monodromy at $ \infty$ in case of the additive middle convolution.
 Roughly speaking the local monodromy at the finite singularities of $L_1\star L_2$ can be computed as follows.
 A Jordan block $\alpha J_k$ of the local monodromy at $t$ of $L_1$ and a Jordan block $\beta J_l$ of the local monodromy at $s$ of $L_2$
contribute (generically)
\[  \begin{array}{ccc}
 \alpha \beta  J_{k} \otimes J_{l}  & \alpha \neq 1, \beta \neq 1 , \alpha \beta \neq 1 \\
    \oplus_{i\in I} J_{i+1} & \alpha \neq 1, \beta \neq 1 , \alpha \beta = 1,& J_k \otimes J_l =\oplus_{i \in I} J_i \\
   \beta J_{k-1} \otimes J_l & \alpha = 1, \beta \neq 1 , \alpha \beta \neq 1 \\
    \alpha J_{k} \otimes J_{l-1} & \alpha \neq 1, \beta = 1 , \alpha \beta \neq 1 \\
    \oplus_{i\in I} J_{i-1} & \alpha = 1, \beta = 1,& J_k \otimes J_l =\oplus_{i \in I} J_i
    \end{array}
\]
 to the  local monodromy at $t+s$ of $L_1\star L_2$.
 For an exact statement see e.g., \cite{DR20}.
\end{rem}

\subsection{Order one operator}

Let $L$ be the operator
with Riemann scheme

\[  \left\{\begin{array}{c| cccc}

            \pm 1 & -1/2\\
            \infty & 1
           \end{array}\right\}. \]

\begin{enumerate}
 \item
The fivefold self convolution of $L$ gives an orthogonal operator of order $5$ (that is the fifth symmetric power) with
Riemann scheme
\[  \left\{\begin{array}{c| ccccc}
            \pm 5 & 0 & 1& 3/2& 2& 3 \\
            \pm 3 & 0 & 1& 3/2& 2& 3 \\
            \pm 1 & 0 & 1& 3/2& 2& 3 \\
            \infty & 1& 1& 1& 1& 1
           \end{array}\right\}.\]

 Since $\Lambda^2 \Sp_4 \cong \SO_5$  we obtain after the pushforward $x \mapsto x^{1/2}$
 (up to scaling and Moebius-transformation) the operator with
Riemann scheme
\[  \left\{\begin{array}{c| ccccc}
 0 & 0& 0& 0& 0&\\
             1/5^2 & 0 & 1/2& 3/2& 2& \\
             1/3^2 &  0 & 1/2& 3/2& 2& \\
             1 &  0 & 1/2& 3/2& 2& \\
            \infty & 3/4& 5/4& 7/4& 9/4&
           \end{array}\right\} \mbox{ (operator \# 6.2) }.\]

On the other hand if we apply the pushforward $x \mapsto x^{-1/2}$ followed by
$MC_{-1}$ we get
 (up to scaling and Moebius transformation) the operator with
Riemann scheme
\[  \left\{\begin{array}{c| ccccc}
            1/5^2 & 0 & 1& 1& 2&  \\
            1/3^2 & 0 & 1& 1& 2&  \\
             1 & 0 & 1& 1& 2&  \\
            0 & 0& 0& 0& 0&\\
             \infty & 1& 1& 2& 2&
           \end{array}\right\}\mbox{ (operator \# 3.1 (AESZ 34) up to Moebius transformation)}.\]

\item
 The sixfold self convolution of $L$ gives a symplectic operator of order $6$ (that is the sixth symmetric power) with
Riemann scheme
\[  \left\{\begin{array}{c| ccccccc}
            t_i & 0 & 1& 2&2 & 3& 4 \\
            \infty & 1& 1& 1& 1& 1 & 1
           \end{array}\right\},\quad t_i \in \{\pm 6, \pm 4, \pm 2,0\}.\]
Applying the pushforward $x \mapsto x^{1/2}$ followed by
$MC_{-1}$ and $\Lambda^2 \Sp_4 \cong \SO_5$ we get operator 6.1.

\end{enumerate}

\subsection{Order two operators}

\begin{enumerate}
\item
Let $L$ be a second order hypergeometric operator with Riemann scheme
\[  \left\{\begin{array}{c| ccc}
    0&-1/2 & -1/2 \\
    1& 0 & 0 \\
    \infty & -a+3/2 & a+1/2
   \end{array} \right\}
\]

The self convolution of $L$
splits into the symmetric and anti symmetric square.
The former has
Riemann scheme
\[  \left\{\begin{array}{c| cccc}
    0 &  0 & 0 & 0 & 0 \\
    1&  0 & 1/2 &  1/2 & 1 \\
    2 &  0 & 1 &  1 & 2 \\
     2/3 &  0 & 1 & 3 & 4 \\
     \infty &  2a & 2-2a & -a+3/2 & a+1/2
   \end{array} \right\}\]

The operator has the form

\[ L=P_0(\vartheta) +\dots + x^5 P_5(\vartheta) \]
where
\[ \begin{array}{ccc}
  P_0(X) & = &   32 X^4 \\
  P_1(X-1/2) & =&
-176 X^4
-24 (4 b+9) X^2
-8 b-7  \\
   P_2(X-1) & = & 376 X^4+
16 (29 b+92) X^2+
64 b^2+80 b+240\\
   P_3(X-3/2) & = &
-388 X^4
-2 (1375+412 b) X^2
-2817/4-570 b-256 b^2 \\
   P_4(X-2) & = &
192 X^4
12 (53 b+161) X^2
336 b^2+828 b+540
\\
    P_5(X-5/2) & = &
-36 X^4
-9 (20 b+49) X^2
-9 (b+2) (16 b+5) \\
b & =& a-a^2
   \end{array}
\]

For special values of $a\in \{ 1/2,1/3,1/4,1/6\}$ we obtain a hypergeometric operator $L$ with monodromy group commensurable to $\SL_2(\ZZ)$ and a Moebius transformation $x \mapsto x t$ yields
Calabi-Yau-operators
with instanton numbers $n_1$ and $n_3$
(the case $a=1$ is reducible but produces also integer instanton numbers)
\[ \begin{array}{c ccccc}
      a & t&   n_1 & n_3 & \# \\
     1/2 &2^6 &  -16& -3280 & 5.9\\
     1/2 & 1/x \cdot 1/2^4    & 4/3 & 44/3& 5.6 \\
    1/3& 2^2 \cdot 3^3& -30  & -14632 &  5.12 \\
    1/4 & 2^8  & -80 & - 174096& 5.13\\
    1/6 & 2^6 \cdot 3^3 & -624 & -43406256 & 5.19 \\
   1 & 2^2 & -2& n_i=0\; (i=2k+1) &
   \end{array}
 \]

\item  Let $L$ be a second order hypergeometric operator with Riemann scheme 
\[  \left\{\begin{array}{c| ccc}
    0& -a+1/2 & a+1/2 \\
    1& 0 & 0 \\
    \infty & 0 & 0
   \end{array} \right\}
\]

The self convolution of $L$
splits into the symmetric and anti symmetric square.
The former has
Riemann scheme
\[  \left\{\begin{array}{c| cccc}
    0 &  -2a & 0 & 0 & 2a \\
    1&  0 & -a+1/2 &  a+1/2 & 1 \\
    2 &  0 & 1 &  1 & 2 \\
     2/3 &  0 & 1 & 3 & 4 \\
     \infty &  1 & 1 & 1 & 1
   \end{array} \right\}\]

The operator has the form
\[ L=P_0(\vartheta) +\dots + x^5 P_5(\vartheta) \]
where
\[ \begin{array}{ccc}
  P_0(X) & = &
32 X^4
-128 X^2 a^2   \\
  P_1(X-1/2) & =&-176 X^4+
16 (38 a^2-15) X^2+
(4 a+1) (4 a-1) (8 a^2+9)
   \\
   P_2(X-1) & = &
376 X^4
-4 (268 a^2-397) X^2
-24 (2 a+1) (2 a-1) (4 a^2+11) \\
   P_3(X-3/2) & = &-388 X^4+
4 (210 a^2-739) X^2
-3451/4+1326 a^2+288 a^4
  \\
   P_4(X-2) & = &192 X^4
-3 (84 a^2-697) X^2
-576 a^2+768

\\
    P_5(X-5/2) & = &-9/4(2X-3)^4
   \end{array}
\]

We get the following Calabi-Yau operators (after a Moebius transformation)

\[ \begin{array}{ccccc}
    a & \# &  \\
    0 & 5.6, 5.9 & \mbox{second MUM point} \\
     1/3 & 5.8 \\
      1/4 & 5.17 \\
      1/6 & 7.9 \\
     \end{array}
 \]

\item The self convolution of a operator   
      with Riemann scheme
      \[\left\{\begin{array}{c| ccccc}
    0 & 0 & 0 \\
    1& 0 & 0 \\
    t & 0 & 0 \\
    \infty & 1 & 1
   \end{array} \right\} \]
      gives
     an operator with symplectic monodromy group with Riemann scheme
      \[\left\{\begin{array}{c| ccccc}
    s & 0 & 1 & 1 & 2 \\
    2/3(t+1)& 0 & 1 & 3 & 4 \\
    \infty & 1 & 1 & 1 & 1 \\
   \end{array} \right\},\quad  s\in \{ 0,1,t,2,1+t,2t\}, \#\{ 0,1,t,2,1+t,2t\}=6. \]
 This gives the Calabi-Yau operators 6.24 and 6.31.

\end{enumerate}

\subsection{Order three operators}

\begin{enumerate}
\item The second exterior power of the self convolution of an orthogonal operator
      with Riemann scheme
      \[\left\{\begin{array}{c| ccccc}
                 t_1 & 0 & 1/2 & 1  \\
                  t_2 & 0 & 1/2 & 1  \\
                  -(t_1+t_2)/2& -a & 0 & a  \\
    \infty & 1 & 1 & 1  \\
               \end{array}\right\}.
       \]
       gives a symplectic operator with Riemann scheme
       \[ \left\{\begin{array}{c| ccccc}
                   0 & 0 & 1 & 3 & 4\\
                 \pm (t_1+t_2) & 0 & 1 & 1 & 2\\
                  \pm (t_1-t_2)/2 & 0 & -a+1/2 & a+1/2 &1  \\
           \infty & 1 & 1 & 1 &1 \\
               \end{array}\right\}.
       \]
       Under the push forward $x\mapsto x^{1/2}$ we obtain the
        Riemann scheme
       \[ \left\{\begin{array}{c| ccccc}
                 0 & 0 & 1/2 & 3/2 & 2\\
                 (t_1+t_2)^2 & 0 & 1 & 1 &2  \\
                 (t_1-t_2)^2/4 & 0 & -a+1/2 & a+1/2 &1  \\
    \infty & 1/2 & 1/2 & 1/2 &1/2 \\
               \end{array}\right\}.
       \]
      Those operators appear as $Q_3$ in \cite[4.3]{BR17} with a different construction.
      \item The second exterior power of the self convolution of an orthogonal operator (being a second symmetric square)
      with Riemann scheme
      \[\left\{\begin{array}{c| ccccc}
                 s & 0 & 1/2 & 1  \\
    \infty & 1 & 1 & 1  \\
               \end{array}\right\}, s\in \{ t_1,t_2,t_3,-(t_1+t_2+t_3),\mid t_i+t_j \neq 0 \mbox{ for } i \neq j\}.
       \]
       gives a symplectic operator with Riemann scheme
       \[ \left\{\begin{array}{c| ccccc}
                 0 & 0 & 1 & 3 & 4\\
                 t & 0 & 1 & 1 &2  \\
    \infty & 1 & 1 & 1 &1 \\
               \end{array}\right\}, t\in \{ \pm (t_1+t_2),\pm (t_2+t_3),\pm (t_1+t_3)\}.
       \]

       Under the push forward $x\mapsto x^{1/2}$ we obtain the
        Riemann scheme
       \[ \left\{\begin{array}{c| ccccc}
                 0 & 0 & 1/2 & 3/2 & 2\\
                 u & 0 & 1 & 1 &2  \\
    \infty & 1/2 & 1/2 & 1/2 &1/2 \\
               \end{array}\right\}, u\in \{ (t_1+t_2)^2,(t_2+t_3)^2,(t_1+t_3)^2\}.
       \]

 This results in the following new operators:

 We start with  second order operators from the Lian and Wiczer list \cite{LW06},
who derived Picard–Fuchs equations for 175 genus-zero subgroups of $PSL(2, \RR)$

 \[ L_{11A} = dx^2+f/g, \quad  f:=1+4 x+46 x^2 +1040 x^3 +6841 x^4 +15628 x^5 +10540 x^6 \]
\[         g:=4 x^2 (-1-2 x+43 x^2 +244 x^3 +416 x^4 )^2
   \]
results in
\[ L_{11A2}= [X^4, -188 X^4-376X^3-322X^2-134X-22, 484(24X^2+48X+19)(X+1)^2,\]
\[-58564(2X+5)(2X+1)(X+2)(X+1)],
\]
with instanton numbers $n_1=6, n_3=200$,
\[L_{14A}=dx^2+f/g,\quad f:=1+4 x+62 x^2 +1036 x^3 +7009 x^4 +18888 x^5 +17688 x^6,\]
\[     g:=4 x^2 (-1-2 x+59 x^2 +256 x^3 +276 x^4 )^2\]
results in
\[L_{14A2}= [X^4, -121 X^4-242 X^3-207 X^2-86 X-14, 196 (22 X^2+44 X+17) (X+1)^2,\]
\[-9604 (2 X+5) (2 X+1) (X+2) (X+1)]
\]
with instanton numbers $n_1=3, n_3=71$ and
\[ L_{15A}=dx^2+f/g,\quad  f:=1+4 x+46 x^2 +1040 x^3 +6841 x^4 +15628 x^5 +10540 x^6 \]
\[         g:=4 x^2 (-1-2 x+43 x^2 +244 x^3 +416 x^4 )^2 \]

results in
\[ L_{15A2}=[X^4, -89 X^4-178 X^3-151 X^2-62 X-10, 100 (32 X^2+64 X+25) (X+1)^2,\]
\[-10000 (2 X+5) (2 X+1) (X+2) (X+1)]
\]
with instanton numbers $n_1=3, n_3=-155/3$.

     A further Hadamard product with $(x-1)^{1/2}$ and the isomorphism $\Lambda^2 \Sp_4 \cong \SO_5$
     gives the Riemann scheme
      \[ \left\{\begin{array}{c| ccccc}
                 0 & 0 & 1 & 1 & 2\\
                 s & 0 & 1/2 & 3/2 &2  \\
    \infty & 1/2 & 1/2 & 1/2 &1/2 \\
               \end{array}\right\}, s\in \{ (t_1+t_2)^2,(t_2+t_3)^2,(t_1+t_3)^2\}.
       \]

This results in the following new operators:
\[ L_{11A3} =[X^4,
2^2(564 X^4+376 X^3+349 X^2+161 X+28), \]
\[ 2^5(70440 X^4+93920 X^3+100519 X^2+53618 X+10768), \]
\[ 2^4(81801152 X^4+163602304 X^3+198502896 X^2+117851728 X+25761639), \]
\[ 2^8 11^2(15708288 X^4+41888768 X^3+56857920 X^2+36723776 X+8504233), \]
\[ 2^{13} 11^4(1001184 X^4+3337280 X^3+5013336 X^2+3470080 X+838693), \]
\[ 2^{14} 11^6(677568 X^4+2710272 X^3+4465680 X^2+3277504 X+819119), \]
\[ 2^{20} 11^8(2 X+1) (4584 X^3+19100 X^2+28814 X+15213), \]
\[ 2^{23} 11^{10}(2 X+1) (2 X+3) (144 X^2+480 X+421), \]
\[ 2^{26} 11^{12}(2 X+1) (2 X+3)^2 (2 X+5)]
\]
with instanton numbers $n_1=60$ and $n_3=38692$, 

\[
L_{14A3}= [X^4,
2(726 X^4+484 X^3+449 X^2+207 X+36),\]
\[ 2^2(227436 X^4+303248 X^3+319295 X^2+166162 X+32632), \]
\[ 2^4 (20069284 X^4+40138568 X^3+47521985 X^2+27301781 X+5802824),\]
\[  2^8  7^2 ( 5572776 X^4+14860736 X^3+19641869 X^2+12282198 X+2770718),\]
\[ 2^{10} 7^4 (3927792 X^4+13092640 X^3+19211440 X^2+12956736 X+3070939),\]
\[ 2^{14} 7^6 (435328 X^4+1741312 X^3+2823224 X^2+2039872 X+504461),\]
\[  2^{ 19}  7^8 (2 X+1) (7260 X^3+30250 X^2+45228 X+23747),\]
\[  2^{ 22} 7^ {10} (2 X+3) (2 X+1) (264 X^2+880 X+771), \]
\[   2^{ 26}  7^{ 12}(2 X+1) (2 X+3)^2 (2 X+5)]
\]

with instanton numbers $n_1=38$ and $n_3=7972$ and

\[L_{15A3}= [X^4,
2(534 X^4+356 X^3+329 X^2+151 X+26),\]
\[ 2^2(133452 X^4+177936 X^3+195803 X^2+108714 X+22424),\]
\[2^4(10135076 X^4+20270152 X^3+25353849 X^2+15564773 X+3479770),\]
\[ 2^6 5^2 (20499456 X^4+54665216 X^3+75946376 X^2+49973136 X+11714505),\]
\[2^{12} 5^4 (1781040 X^4+5936800 X^3+9043104 X^2+6311232 X+1533013),\]
\[2^{16} 5^6 (423680 X^4+1694720 X^3+2808784 X^2+2067232 X+517343),\]
\[2^{23} 5^8(2 X+1) (4140 X^3+17250 X^2+26064 X+13765),\]
\[2^{27} 5^{10}(2 X+3) (2 X+1) (192 X^2+640 X+561),\]
\[2^{32} 5^{12}(2 X+1) (2 X+3)^2 (2 X+5)]
\]

with instanton numbers $n_1=30$ and $n_3=8140/3$.

     \item The second exterior power of the self convolution of an orthogonal operator
      with Riemann scheme (e.g., obtained by the middle convolution of a Beauville operator with $\MC_{-1}$
      followed by a pullback $x \mapsto x^2$)
      \[\left\{\begin{array}{c| ccccc}
                 s & 0 & 1/2 & 1  \\
    \infty & 1 & 1 & 1  \\
               \end{array}\right\}, s\in \{ \pm t_1,\pm t_2 \}.
       \]
       gives a symplectic operator with Riemann scheme
       \[ \left\{\begin{array}{c| ccccc}
                 0 & 0 & 1 & 1 & 2\\
                 t & 0 & 1 & 1 &2  \\
    \infty & 1 & 1 & 1 &1 \\
               \end{array}\right\}, t\in \{ \pm (t_1 \pm t_2) \}.
       \]

       Under the push forward $x\mapsto x^{1/2}$ we obtain the
        Riemann scheme
       \[ \left\{\begin{array}{c| ccccc}
                 0 & 0 & 1/2 & 1/2 & 1\\
                 u & 0 & 1 & 1 &2  \\
    \infty & 1/2 & 1/2 & 1/2 &1/2 \\
               \end{array}\right\}, u\in \{ (t_1\pm t_2)^2 \}.
       \]

 This results in the  operators
$ Q_2$ where $\lambda=0$ in \cite[4.3]{BR17}.

     A further Hadamard product with $(x-1)^{1/2}$ and the isomorphism $\Lambda^2 \Sp_4 \cong \SO_5$
     gives the Riemann scheme
      \[ \left\{\begin{array}{c| ccccc}
                 0 & 0 & 1 & 1 & 2\\
                 s & 0 & 1/2 & 3/2 &2  \\
    \infty & 1/2 & 1/2 & 1/2 &1/2 \\
               \end{array}\right\}, s\in \{ (t_1+t_2)^2,(t_2+t_3)^2,(t_1+t_3)^2\}.
       \]

This results in the  operators $ Q_3$, where $\lambda=0$, in \cite[4.3]{BR17}.

 \item The second symmetric power of the self convolution of an orthogonal operator
      with Riemann scheme  (e.g., the middle convolution of a Beauville operator with $x^{1/2}$)
      \[\left\{\begin{array}{c| ccccc}
                 s & 0 & 1/2 & 1  \\
    \infty & 1/2 & 1/2 & 1/2  \\
               \end{array}\right\}, s\in \{ t_1,t_2,t_3\}.
       \]
       gives a symplectic operator with Riemann scheme
       \[ \left\{\begin{array}{c| ccccc}
                 t & 0 & 1 & 3 & 4\\
                 u & 0 & 1 & 1 &2  \\
    \infty & 1 & 1 & 1 &1 \\
               \end{array}\right\}, u\in \{ t_i+t_j,\; i,j=1,2,3 \}.
       \]
   But this is the same operator as in the case  of  the self convolution of a Beauville operator (second symmetric square).
\end{enumerate}

\subsection{Order four operators}

\begin{enumerate}
\item The self convolution of a symplectic operator
      with Riemann scheme
       \[\left\{\begin{array}{c| ccccc}
    s & 0 & 1 & 1 & 2 \\
    \infty & 1/2 & 1/2 & 1/2 & 1/2 \\
   \end{array} \right\},\quad  s\in \{ 0,1,t_1,t_2\} \]
  splits in a symmetric and anti-symmetric square.

  The latter is
    a symplectic operator
      with Riemann scheme
       \[\left\{\begin{array}{c| ccccc}
    s & 0 & 1 & 1 & 2 \\
    \infty & 1/2 & 1/2 & 1/2 & 1/2 \\
   \end{array} \right\},\quad  t\in \{ 0,1,t_1,1+t_1, 1+t_2,t_1+t_2\}. \]

 This gives the Calabi-Yau operators as in the case for the order three operators.

\item The self convolution of a symplectic operator
      with Riemann scheme
       \[\left\{\begin{array}{c| ccccc}
     0 &   0 & a & 1-a & 1 \\
    s & 0 & 1 & 1 & 2 \\
    \infty & 1/2 & 1/2 & 1/2 & 1/2 \\
   \end{array} \right\},\quad  s\in \{ t_1,t_2\} \]
  splits in a symmetric and anti-symmetric square.

  The latter is
    a symplectic operator
      with Riemann scheme
       \[\left\{\begin{array}{c| ccccc}
    0 & 0 & 1 & 1 & 2 \\
    s & 0 & a & 1-a & 1 \\
     t_1+t_2 & 0 & 1 & 1 & 2 \\
       (t_1+t_2)/2 & 0 & 1 & 3 & 4 \\
    \infty & 1 & 1 & 1 & 1 \\
   \end{array} \right\},\quad  s\in \{ t_1, t_2\}. \]

After a Moebius transformation and a push forward $x \mapsto x^{1/2}$ we get   the Riemann scheme

   \[\left\{\begin{array}{c| ccccc}
    0 & 0 & 1/2 & 3/2 & 2 \\
    (t_1-t_2)^2/4 & 0 & a & 1-a & 1 \\
     (t_1+t_2)^2/4 & 0 & 1 & 1 & 2 \\
    \infty & 1 & 1 & 1 & 1 \\
   \end{array} \right\}. \]

 This results in the operators
$ Q_3$  in \cite[4.3]{BR17}.

\end{enumerate}

\subsection{Pullbacks of hypergeometric operators  and the Katz-Arinkin-algorithm}

Here we use the following construction.
  Let  $L$ be a $n$-th  order hypergeometric operator with Riemann scheme
\[  \left\{\begin{array}{c| cccccccc}
    0& -a_1 & \cdots & -a_{n/2} & a_{n/2} & &\cdots & a_1\\
    1& 0 &  \cdots & n/2-1 & n/2-1 & n/2 & \cdots & n-2 \\
    \infty & 1 & \cdots & & &&\cdots   &1
   \end{array} \right\}
\]
for $n$ even (hence having symplectic monodromy) and
\[  \left\{\begin{array}{c| cccccccc}
    0&-a_1 & \cdots &  -a_{(n-1)/2}&1/2 & a_{(n-1)/2} & \cdots & a_1 \\
    1& 0 & \cdots & (n-1)/2 & n/2 & (n+1)/2 & \cdots & n-1 \\
    \infty & 1 && &\cdots& && 1
   \end{array} \right\}
\]
  for $n$ odd (hence having orthogonal monodromy).

  Let further $\pi_k$ be the pullback $x \mapsto x^k$  and   $\pi^k$ be the push forward $x \mapsto x^{1/k}$
  and
  \[  \phi:= \pi^k \circ  FT^{-1} \circ  \Gamma \circ  FT \circ \pi_k, \]
  where $FT$ denotes the formal Fourier-Laplace transform, $FT^{-1}$ its inverse and $\Gamma \in \{ \Sym^l, \Lambda^l\}$.
  Finally we apply the Hadamard product
   \[ \phi (L)  \star L_0, \]
  where
  $L_0$ is a suitable hypergeometric operator with self dual monodromy group defined over $\ZZ$.

  \begin{rem}
   The Fourier transform of $\pi_k ( L) $ has an  irregular singularity of pure  slope $1$ at infinity.
   Hence the push-forward
   $\pi^k$ of $(FT \circ \pi_k)(L)$ has an  irregular singularity of pure  slope $1/k$ at infinity.
   Taking exterior or symmetric powers does not alter the slope.
   Thus after applying the Katz-Arinkin algorithm \cite{Ar10}  $k$ times we obtain a  differential operator of slope $1$,
   i.e. again a fuchsian operator.
   Down to earth we just apply the sequence $ (x \mapsto 1/x) \circ FT$ $k$-times for almost all of our examples.
   In each step the local monodromy at the regular singularity $0$ is obtained by reducing the size of each unipotent Jordan block
   of the local monodromy at $0$ by $k$ and leaving the size all others Jordan blocks unchanged.
   This results in the same operator as above for $\phi (L)  \star L_0.$
  \end{rem}

We start with the following list of hypergeometric operators $L_k$ and apply the above construction.

  \[ \small  \begin{array}{cccc}
       L_k &       &    \pi^k \circ FT \circ \pi_k (L_k)\\
       \hline
      L_2 & [ X + 1/2, -(X+1)]& [2^2  X^2, -1]\\
      L_{3a}&       [ \prod_{i=1}^2(X+i/3) ,-(X+1)^2] & [3^3  X^3, -1]\\
       L_{3b} & [  (X+1/2) \prod_{i=1}^2 (X+i/3) ,-(X+1)^3] &[2\cdot 3^3  X^4, -1-2  X]\\
       L_{3c} &  [(X+1/2)^2 \prod_{i=1}^2 (X+i/3),-(X+1)^4]&[2^2 \cdot 3^3  X^5, -(2  X+1)^2] \\
       L_{3d}& [\prod_{i=1}^2(X+i/3)^2 ,-(X+1)^4]& [3^5  X^5, -(3  X+2)  (3  X+1)]\\
       L_{3e}& [ (X+1/6)  (X+5/6) \prod_{i=1}^2(X+i/3)    ,-(X+1)^4] & [2^2 3^5  X^5, -(6  X+5)  (6  X+1)] \\
       L_{3f} & [   (X+1/4)  (X+3/4)\prod_{i=1}^2(X+i/3)   ,-(X+1)^4]& [2^4 3^3  X^5, -(4  X+3)  (4  X+1)] \\
       L_{4a} & [(X+1/4)  (X+3/4),-(X+1)^2] &[2^7  X^3  (2  X-1), -1]\\
       L_{4b} & [(X+1/2)^2  (X+1/4)  (X+3/4),-(X+1)^4] & [2^8  X^5, -1-2  X]\\
        L_{5a}&  [ \prod_{i=1}^4(X+i/5) ,-(X+1)^4] &  [5^5  X^5, 1]\\
         L_{5b}&  [ \prod_{i \in \{-2,-1,1,2\}} (X+i/5) ,-(X+1/2)^4] & [5^5  X^4  (2  X-1), -2]\\
         L_{5c}&    [(X+1/2)\prod_{i=1}^4(X+i/5) ,-(X+1)^5] & [2 \cdot  5^5  X^6, -1-2  X] \\
         L_{6a}& [\prod_{i=1}^5(X+i/6),-(X+1)^5]& [6^6  X^6, -1]\\
         L_{6b}& [ (X+1/2) \prod_{i=1}^5(X+i/6),-(X+1)^6] & [2^73^6  X^7, -1-2  X]
\\
         L_7 &[ \prod_{i=1}^6(X+i/7),-(X+1)^6] & [7^7  X^7, -1]
\\

     \end{array}
\]

\[\begin{array}{cccccc}
   L_k  & \Sym^l/\Lambda^l &  L_0 &\mbox{remarks} & \#\\
   L_2 & \Sym^5 &  L_{2}(1/x)  & &3.1 \\
       &     \Sym^5  & - &  \Lambda^2 \Sp_4 \cong \SO_5  & 6.2& \\
       &    \Sym^6  &    L_{2}(1/x) &   \Lambda^2 \Sp_4 \cong \SO_5 &  6.1& \\
      L_{3a}  & \Sym^3  &    L_{3a}(1/x)&   \mbox{ second MUM point } & 5.1, 5.71 &  \\
        L_{3b}  & \Sym^2  &  L_{3a}(1/x)&  \mbox{ second MUM point } &5.4, 5.11  &  \\
         L_{3c}   &  \Lambda^2  & L_{3a}(1/x)  & \mbox{ second MUM point }  &5.47, 5.48 & \\
       L_{3d} &     \Lambda^2 &  L_{3a}(1/x)& \mbox{ second MUM point } &5.58, 5.59 & \\
         L_{3e}&      \Lambda^2 &  L_{3a}(1/x)& \mbox{ second MUM point }  & 5.60, 5.61 & \\
           L_{3f}  &  \Lambda^2   &   L_{3a}(1/x)& \mbox{ second MUM point } & 5.62, 5.63 &  \\
       L_{4a} & \Sym^2   & L_{4a}(1/x)& \mbox{ second MUM point }  & 4.33, 4.56  &  \\
        L_{4b}   &  \Lambda^2  &  L_{4a}(1/x)& & 2.60 &  \\
    L_{5a} & \Sym^2   &  L_{5a}(1/x)& \mbox{ second MUM point }   & 5.5, 5.16 &  \\
       L_{5b} &  \Lambda^2 &  L_{5b}(1/x) & \mbox{ second MUM point }   & 4.50, 4.51 &  \\
          L_{5c}   &  \Lambda^2 &L_{5a}(1/x)&  \mbox{ second MUM point }  & 5.15,  5.31 & \\
     L_{6a} & \Lambda^3  & L_{6}(1/x)  & &2.62     & \\
     L_{6b} &   \Lambda^2   &   L_{6}(1/x) &  \mbox{$FT \circ \pi_6$ is $G_2$ operator} & 4.28 \\
 L_7 &\Lambda^2   &   L_{7}(1/x)& \mbox{ second MUM point }  & 5.7, 5.46
  \end{array}
\]

\begin{rem}
\begin{enumerate}
 \item
The differential Galois group of the operator
 $\vartheta^{2n}-x$ is $Sp_{2n}$, of
  $\vartheta^{2n+1}+x$ is $SL_{2n+1}$ and of
 $\vartheta^{2n+1}-x(2\vartheta+1)$ is $\SO_{2n+1}$ for $n\neq 3$ and
 $G_2$ for $n=3,$ see \cite{Katz90}, \cite[Section 6]{FrGr}.

 \item

Note that that some of the above operators appear for
Calabi-Yau complete intersections 3-folds $X$ in Grassmannians $G(k, n)$ \cite[Section 5 and 7]{BFKvS98}.
\[\begin{array}{ccccc}
G(2,5) & =&\Lambda^2([X^5,-1)]) \\
G(2,6) & =&\Lambda^2([X^6,1]) \\
G(2,7) &  &\Lambda^2([X^7,1])  \mbox{ is a factor of  } G(2,7)  \\
G(3,6) & = &\Lambda^3([X^6,1]) \\
\end{array}
\]

 \end{enumerate}
\end{rem}

We demonstrate the construction for the following examples.

 \begin{ex}
 \begin{enumerate}
\item   Let  $L=L_{5c}$ be the $5$-th  order hypergeometric operator with Riemann scheme
\[  \left\{\begin{array}{c| cccccccc}
    0& 1/5 & 2/5 & 1/2 & 3/5 & 4/5 \\
    1& 0 & 1 & 3/2  & 2 & 3 \\
    \infty & 1 & 1 & 1 &1 & 1
   \end{array} \right\}
\]
 and $L_0$ be the  fourth  order hypergeometric operator with Riemann scheme
\[  \left\{\begin{array}{c| cccccccc}
    \infty & 1/5 & 2/5 & 3/5 & 4/5 \\
    1& 0 & 1 & 1  & 2  \\
    0 & 0 & 0 & 0 & 0
   \end{array} \right\}
\]

  Then
  \[    ((\pi^5 \circ  FT^{-1} \circ  \Lambda^2 \circ FT \circ \pi_5)  (L)) \star L_0 \]
  gives
    the Riemann-scheme
  \[  \left\{\begin{array}{c| cccccccc}
    \infty & 1 & 1 & 1 & 1 \\
     0 & 0 & 0 & 0 & 0 \\
    1& 0 & 1 & 1  & 2  \\
    (1+\zeta_5)^5=-\frac{11+5 \sqrt{5}}{2} & 0 & 1 & 1 & 2 \\
     (1+\zeta_5^2)^5=-\frac{11-5 \sqrt{5}}{2} & 0 & 1 & 1 & 2 \\
     3/14 & 0 & 1 & 3 & 4
   \end{array} \right\}
\]
  After a Moebius transformation we get
  the operators   \# 5.31 (AESZ: 212) and  \# 5.15 (AESZ: 117)

 \item  Let  $L=L_{5b}$ be the fourth  order hypergeometric operator with Riemann scheme
\[  \left\{\begin{array}{c| cccccccc}
    0& -2/5 & -1/5 & 1/5 & 2/5 \\
    1& 0 & 1 & 1  & 2  \\
    \infty & 1/2 & 1/2 & 1/2 &1/2
   \end{array} \right\}.
\]

  Then
   \[    ((\pi^5 \circ  FT^{-1} \circ \Lambda^2 \circ FT \circ \pi_5)  (L)) \star L \]
  gives
  the Riemann-scheme
  \[  \left\{\begin{array}{c| cccccccc}
     0 & 0 & 0 & 0 & 0 \\
    (1+\zeta_5)^5=-\frac{11+5 \sqrt{5}}{2} & 0 & 1 & 1 & 2 \\
     (1+\zeta_5^2)^5=-\frac{11-5 \sqrt{5}}{2} & 0 & 1 & 1 & 2 \\
     -1/2 & 0 & 1 & 3 & 4 \\
     \infty& 1/2 & 1/2 & 1/2 & 1/2
   \end{array} \right\}.
\]
 Up to scaling and a Moebius transformation  we obtain the operators   \# 4.50  and  \# 4.51.

\end{enumerate}
\end{ex}

\begin{proof}
 We outline the claim for the monodromy using Remark~\ref{FT}.
 \begin{enumerate}
  \item The local monodromy  of
        $L(x^5)$ is
        \[  \left\{\begin{array}{c| cccccccc}
                     0 & -1 \oplus J_1^4 \\
                     \zeta_5^i & -J_1 \oplus J_1^4 \\
                     \infty & J_5
                   \end{array}\right\}, i=1,\ldots,5,\; \zeta_5^5=1. \]
       $FT(L(x^5))$ has order $6$ with local monodromy $J_6$ at zero.
       Hence the second exterior power of it has local monodromy $\Lambda^2 J_6=J_1\oplus J_5 \oplus J_9$ at $0$.
       Applying $FT^{-1}$ we get an order $12$ operator with   local monodromy $ J_4 \oplus J_8$ at $\infty$.
       The further
       singularities are  $0+\zeta_5^i, \zeta_5^i+\zeta_5^j,\; i,j=1,\ldots,5,i<j, $ with local monodromy $J_2\oplus J_1^{10}$
       each.
       Hence the push-forward $x\mapsto x^{1/5}$ has the singularities $1,(1+\zeta)^5,(1+\zeta^2)^5$ with local monodromy $J_2\oplus J_1^{10}$
       each, $0$ with local monodromy $\oplus_{i=1}^4 \zeta_5^{i} J_1^3$ and $ J_4 \oplus J_8$ at $\infty$.
       Therefore the Hadamard product with $L_0=L(1/x)$ gives the claim by Section~\ref{MC}.
     \item   The local monodromy  of
        $L(x^5)$ is
        \[  \left\{\begin{array}{c| cccccccc}
                      \zeta_5^i &  J_2 \oplus J_1^2 \\
                     \infty & J_4
                   \end{array}\right\}, i=1,\ldots,5,\; \zeta_5^5=1. \]
       $FT(L(x^5))$ has order $5$ with local monodromy $-J_4\oplus J_1$ at zero.
       Hence the second exterior power of it has local monodromy $\Lambda^2 J_5=-J_4\oplus J_1 \oplus  J_5 $ at $0$.
       Applying $FT^{-1}$ we get an order $8$ operator with   local monodromy $ -J_4 \oplus J_4$ at $\infty$.
       The further
       singularities are  $\zeta_5^i+\zeta_5^j,\; i,j=1,\ldots,5,i<j, $ with local monodromy $J_2\oplus J_1^{6}$
       each.
        Hence after the push-forward $x\mapsto x^{1/5}$ we have the singularities $(1+\zeta_5)^5,(1+\zeta_5^2)^5$ with local monodromy $J_2\oplus J_1^{6}$
       each, $0$ with local monodromy $\oplus_{i=1}^4 \zeta_5^{i} J_1^2$ and $ J_4 \oplus -J_4$ at $\infty$.
       Thus the Hadamard product with $L(1/x)$ gives the Calabi-Yau operator up to a Moebius transform.
 \end{enumerate}

\end{proof}

\begin{rem}

A similar  construction  as above with
\[  \pi^{2k}  (  (FT\circ \pi_k)  (L) \otimes  (FT\circ \pi_k)  L(-x))) \]
yields the operators 4.39, 4.40 (second MUM point), where
\[ L=[(X+1/2)(X+1/4)(X+3/4),-(X+1)^3],\quad \pi_k=x^4. \]
\end{rem}

\section{Miscellaneous}

\begin{rem}
 \begin{enumerate}
  \item
Let $L$ be a fourth order symplectic operator with Riemann scheme
\[   \left\{\begin{array}{c| cccccccc}
    \infty &  a & 1 & 1 & 1- a \\
    1& 0 & 1 & 1  & 2  \\
     t& 0 & 1 & 1  & 2  \\
    0 & 0 & 0 & 0 & 0
   \end{array} \right\},  a \not \in \ZZ, a \not \in 1/2+\ZZ.
\]
Then
$\MC_{\alpha}, \alpha=\exp(2 \pi i a),$ yields
a fourth order  operator with Riemann scheme
\[   \left\{\begin{array}{c| cccccccc}
    \infty & 1- a & 1- a & 1- a & 1-2 a \\
    1& 0 & 1 & 2  & 1+ a  \\
     t& 0 & 1 & 2  & 1+ a  \\
    0 & 0 &  a &  a &  a
   \end{array} \right\}.
\]
A Moebius transform gives
\[   \left\{\begin{array}{c| cccccccc}
      \infty & 1-2 a & 2-2 a & 3-2 a & 4-2 a \\
    \pm t & 0 &  a &  a &  a \\
    \pm 1 & 0 & 1 & 2  & 1+ a  \\
   \end{array} \right\}.
\]
A pushforward $ x \mapsto x^{1/2}$ results in
\[   \left\{\begin{array}{c| cccccccc}
      \infty & 1/2- a & 1- a & 3/2- a & 2- a \\
     t^2 & 0 &  a &  a &  a \\
     1 & 0 & 1 & 2  & 1+ a  \\
     0 & 0 & 1/2 & 1 & 3/2
   \end{array} \right\}.
\]
Thus both
$\MC_{-\alpha}$ and $\MC_{-\alpha} \circ  \mu_{x^{1/2}}$
yields a fourth order symplectic (up to scaling)  operator,
where $\mu_{x^{1/2}}$ denotes the multiplication by $x^{1/2}$.

Thus the  Riemann scheme is
\[   \left\{\begin{array}{c| cccccccc}
      \infty & 1/2 & 1+ a & 3/2 & 2+ a \\
     t^2 & 0 & 0 & 0 & 0 \\
     1 & 0 & 1 & 1  & 2  \\
     0 & 0 & 1/2- a & 1 & 3/2- a
   \end{array} \right\},
\]
in the former case and in the latter case
\[   \left\{\begin{array}{c| cccccccc}
      0 & -1/2- a & 0 & -3/2- a & 1 \\
     t^2 & 0 & 0 & 0 & 0 \\
     1 & 0 & 1 & 1  & 2  \\
     \infty & 1+ a & 1/2+1 & 2+ a & 1+3/2
   \end{array} \right\}, \mbox{resp..}
\]

This construction yields the following Calabi-Yau-operators after a suitable Moebius transform

\[ \begin{array}{cccc}
    L &  L_1 & L_2 \\
     2.62 (a=1/4) & 3.28 &3.29 \\   
   \end{array}
\]

  \item There is a  similar construction  starting with an orthogonal fifth order operator with Riemann scheme
         (arising as second exterior power of a symplectic fourth order operator)
       \[   \left\{\begin{array}{c| cccccccc}
    \infty &  a & 1 & 1 & 1 & 1- a \\
    1& 0 & 1 & 3/2 &2  & 3  \\
     t& 0 & 1 & 3/2 &2  & 3  \\
    0 & 0 & 0 & 0 & 0 & 0
   \end{array} \right\},  a \not \in \ZZ,  a \not \in 1/2+\ZZ
\]
 We trace the change of the Riemann scheme:
    \[  \MC_{\alpha}, \alpha=\exp(2 \pi i a) :  \left\{\begin{array}{c| cccccccc}
    \infty & 1-2 a & 1- a & 1- a & 1- a & 1- a \\
    1& 0 & 1 & 3/2+ a &2  & 3  \\
     t& 0 & 1 & 3/2+ a &2  & 3  \\
    0 &  a &  a &  a &  a & 0
   \end{array} \right\}
\]
After scaling and a Moebius-transform we get
 \[  
  \left\{\begin{array}{c| cccccccc}
   0& 0& 1/2 & 1 & 3/2 &2   \\
  1 &  a &  a &  a &  a & 0\\
   s&  0 & 1 & 3/2+ a &2  & 3  \\
   \infty &1/2- a& 1- a & 3/2- a & 2- a &3/2-  a    \\
  \end{array} \right\}.  \]
Then
  \[
 \MC_{-\alpha}:  \left\{\begin{array}{c| cccccccc}
    \infty & 3/2 & 3/2+ a & 5/2 & 5/2+ a \\
    s& 0 & 1 & 1 &  2  \\
     0 & - a & 0 & 1- a &1    \\
    1 & -1/2 & -1/2 & -1/2 & -1/2
   \end{array} \right\}.
\]
After a Moebius-transform
\[
   \left\{\begin{array}{c| cccccccc}
   0 & 0 & 0 & 0 & 0 \\
     t_1& - a/2 &  a/2 & 1- a/2 &1+ a/2    \\
    t_2 & - a/2 &  a/2 & 1- a/2 &1+ a/2    \\
     \infty & 1 & 2 & 2 &  3  \\
   \end{array} \right\}.
\]
If $(x-t_1)(x-t_2) \in \ZZ[x]$ we get a possible candidate for a Calabi-Yau operator $L_A$.
 \end{enumerate}

From this construction we obtain

\[ \begin{array}{ccccc}
   & L &     L_A  \\
   &4.25 (a=1/3) & 4.75         \\    
   \end{array}
\]

However we don't know whether operator 4.25 is of geometric origin.
\end{rem}

\section{Hadamard Products}

\subsection{Beauville operators}
\begin{rem} 
 Let $L_1, L_2$ be Beauville operators, i.e. second order operators with unipotent monodromy at each of its four singularities arising from  families of elliptic curves with four singular fibres.
 Their Hadamard product $L_1 \star L_2$ is a Calabi-Yau operator  (for $L_1$ and $L_2$ not being isomorphic one gets the  operators \#8.1,\ldots,\#8.15  see \cite{AZ06}).
 It turns out that they are
 invariant under $x \mapsto 1/x$ after a suitable Moebius transformation.
Hence we can apply the push forward $x \mapsto x^{1/2}$   after a further suitable   Moebius transformation.
Thus the Riemann scheme of the resulting (reduced) operators looks like
  \[\left\{\begin{array}{c| ccccc}
    0 & 0 & 0 & 0 & 0 \\
    t_1 & 0 & 1 & 1 & 2 \\
    t_2 & 0 & 1 & 1 & 2 \\
     t_3 & 0 & 1/2 & 3/2 & 2 \\
      t_4 & 0 & 1/2 & 3/2 & 2 \\
    \infty & 1/2 & 3/2 & 5/2 & 7/2 \\
   \end{array} \right\},\quad  (x-t_1)(x-t_2),(x-t_3)(x-t_4) \in \QQ[x]. \]
 This refines the claim in the database \cite{AESZ}, where these operators    are only defined over an quadratic extension
 of $\QQ$.\\

However the following construction was missed in \cite{AZ06}.
 The Beauville operator of type
 $I_1 I_1 I_5 I_5$ arising from families of elliptic curves, see \cite{Herf91},
 $L_{I_1 I_1 I_5 I_5}=[X^2, 11X^2+11X+3, -(X+1)^2]$ has the Riemann-scheme
 \[ \left\{\begin{array}{c|ccc}
     0 & 0 & 0 \\
     11/2-\frac{5 \sqrt{5}}{2} &  0 & 0 \\
      11/2+\frac{5 \sqrt{5}}{2} &  0 & 0 \\
       \infty & 1 & 1 \\
    \end{array}\right\}
 \]
 After a Moebius transformations we obtain  two Galois conjugate   operators $\tilde{L}_{I_1 I_1 I_5 I_5}$
   and $\tilde{L}^\sigma_{I_1 I_1 I_5 I_5}$
 with Riemann-scheme
 \[ \left\{ \begin{array}{c|ccc}
     0 & 0 & 0 \\
     1 &  0 & 0 \\
      1/2+ \frac{11 \sqrt{5}}{50} &  0 & 0 \\
       \infty & 1 & 1 \\
    \end{array} \right\}\mbox{ and }
    \left\{\begin{array}{c|ccc}
     0 & 0 & 0 \\
     1 &  0 & 0 \\
      1/2-\frac{11 \sqrt{5}}{50} &  0 & 0 \\
       \infty & 1 & 1 \\
    \end{array} \right\}.
 \]
 Thus their Hadamard product gives the Calabi-Yau operator  $L$ with Riemann scheme
  \[  \left\{\begin{array}{c|ccccc}
     0 & 0 & 0 & 0 & 0\\
     1 &  0 & 1 & 1 & 2 \\
      1/125 &  0 & 1 & 1 & 2 \\
     1/2\pm \frac{11 \sqrt{5}} {50} &  0 & 1 & 1 & 2 \\
     \pm  \frac{\sqrt{5}}{25} &  0 & 1 & 3 & 4  \\
       \infty & 1 & 1 & 1 & 1 \\
    \end{array}\right\},
 \]
where
 \[L=[ X^4, -251 X^4-502 X^3-411 X^2-160 X-25, 5^3  (126 X^4+516 X^3+716 X^2+400 X+81), \]
 \[5^3  (251 X^4-1506 X^3-4609 X^2-3492 X-935), 5^6 (-254 X^4-508 X^3+394 X^2+648 X+199),\] \[5^6 (251 X^4+2510 X^3+1415 X^2-204 X-295), 5^9 (126 X^4-12 X^3-76 X^2-12 X+7),\] \[5^9 (-251 X^4-502 X^3-411 X^2-160 X-25),
5^{12} (X+1)^4]\]
has the instanton numbers $n_1=21$ and $n_3=595$.
After a Moebius transformation and the pushforward we obtain up to a Moebius transformation
the reduced Calabi-Yau operator $\tilde{L}$ with Riemann scheme
\[
   \left\{\begin{array}{c| ccccc}
    0 & 0 & 0 & 0 & 0 \\
    -1 & 0 & 1 & 1 & 2 \\
    -63/7688\pm 5 \sqrt{5}/7688  & 0 & 1/2 & 3/2 & 2 \\
    \infty & 1/2 & 3/2 & 3/2 & 5/2 \\
   \end{array} \right\}, \]
  where
\[\tilde{L}= [X^4, 757 X^4+506 X^3+471 X^2+218 X+38, 2^2  (59349 X^4+79384 X^3+81795 X^2+41110 X+7841),\]
\[ 2^2 (9872040 X^4+19822960 X^3+22492365 X^2+12147007 X+2448500), \]
\[ 2^4 (229864260 X^4+616242320 X^3+765625685 X^2+438007942 X+91428551),\]
\[2^4 31^2 (2 X+1) (5930448 X^3+16960696 X^2+18484142 X+6931643),\]
\[ 2^6 31^4 (2 X+3) (2 X+1) (16132 X^2+33272 X+17953), 2^8 31^6 (2 X+1) (2 X+3)^2 (2X+5)
]. \]
\end{rem}

\subsection{Hadamard products of dual second order operators}

Let
$L$ be a second order operator with Riemann scheme

\[  \left\{\begin{array}{c| ccccc}
    0 & 0 & 0  \\
    1 & 0 & a & \\
     t & 0 & -a & \\
    \infty & 1 & 1
 \end{array} \right\}.\]

 Then for $t\neq -1$ the Hadamard product of $L$ with $L(1/x)=L^\vee(xt),$ 
 where $L^\vee$ denotes the dual operator, is a fourth order symplectic operator with Riemann scheme
 \[
   \left\{\begin{array}{c| ccccc}
    0 & 0 & 0 & 0 & 0 \\
    1 & 0 & 1-2a & 1+2a & 2 \\
    -1 & 0 & 1 & 3 & 4 \\
    t & 0 & 1 & 1 & 2 \\
   1/t & 0 & 1 & 1 & 2 \\
    \infty & 1 & 1 & 1 & 1 \\
  \end{array} \right\}.\]
A Moebius transformation gives after scaling
\[
   \left\{\begin{array}{c| ccccc}
     0 & 0 & 1 & 3 & 4 \\
    \pm 1 & 0 & 0 & 0 & 0 \\
    \pm (t+1)/(t-1) & 0 & 1 & 1 & 2 \\
    \infty & 1 & 2-2a & 2+2a & 3 \\
  \end{array} \right\}.\]
The push-forward $x \mapsto x^{1/2}$ and a  Moebius transformation finally yields
\[
   \left\{\begin{array}{c| ccccc}
     0 & 0 & 0 & 0 & 0 \\
     1 & 0 & 1/2 & 3/2 & 2 \\
    ((t+1)/(t-1))^2 & 0 & 1 & 1 & 2 \\
    \infty & 1/2 & 1-a & 1+a & 3/2 \\
  \end{array} \right\}.\]

 Starting with with one of the operators $L$ arising from families of elliptic curves \cite{Herf91} we get
  \[ \begin{array}{cc}
      \mbox{\cite{Herf91}} & \#\\
      I_1\; I_7\; II\;II &  3.20     \\
       I_2\; I_6\; II\;II &   2.58    \\
        I_1\; I_6\; II\;II &    3.15   \\
        I_1\; I_5\; III\;III &    3.19   \\
           I_2\; I_4\; III\;III &  2.17     \\
              I_3\; I_3\; III\;III &    3.7   \\
              I_2\; I_2\; IV\;IV &    2.59   \\
     \end{array}
\]
  These operators also appear in \cite[Section 4.3 family $Q_3$]{BR17} using a different construction.
  However we also start with the same second order two operators $L$ as in  \cite[Section 4.2]{BR17}
   ${\rm sign}(L)=(0,\lambda,\lambda,0),\quad \lambda=a$.

  For $t= -1$ the Hadamard product of $L$ with $L(1/x)$ is a fourth order operator with Riemann scheme
 \[
   \left\{\begin{array}{c| ccccc}
    0 & 0 & 0 & 0 & 0 \\
    1 & 0 & 1-2a & 1+2a & 2 \\
    -1 & 0 & 1 & 1 & 2 \\
    \infty & 1 & 1 & 1 & 1 \\
  \end{array} \right\}\]
A Moebius transformation gives
\[
   \left\{\begin{array}{c| ccccc}
    \pm 1 & 0 & 0 & 0 & 0 \\
    0 & 0 & 1 & 1 & 2 \\
    \infty & 1 & 2-2a & 2+2a & 3 \\
  \end{array} \right\}.\]
The push-forward $x \mapsto x^{1/2}$ and a  Moebius transformation gives
\[
   \left\{\begin{array}{c| ccccc}
     0 & 0 & 0 & 0 & 0 \\
     1 & 0 & 1/2 & 1/2 & 1 \\
    \infty & 1/2 & 1-a & 1+a & 3/2 \\
  \end{array} \right\}.\]
  Being symplectically rigid these operators already appear in \cite{BR13}.

  \bibliographystyle{amsplain}

\end{document}